\documentclass{article}
\usepackage{amsfonts}
\usepackage{latexsym}
\usepackage{amsmath}
\usepackage{amsthm}
\usepackage{amssymb}
\usepackage{amscd}
\usepackage{epsfig}
\usepackage{graphicx, subfigure}
\addtocounter{MaxMatrixCols}{4}
\newtheorem{theorem}{Theorem}[section]
\newtheorem{proposition}[theorem]{Proposition}
\newtheorem{lemma}[theorem]{Lemma}

\theoremstyle{definition}
\newtheorem{definition}[theorem]{Definition}

\newtheorem{conjecture}[theorem]{Conjecture}
\theoremstyle{remark}
\newtheorem*{remark}{Remark}
\renewcommand{\qed}{\hskip 10pt $\Box$}
\renewenvironment{proof}{\par {\sc {\bf Proof.}\hskip 5pt}}{\hfill \qed \par}

\begin{document}
\title{Genetic theory for cubic graphs}
\author{Pouya Baniasadi, Vladimir Ejov, Jerzy A. Filar, Michael Haythorpe}
\maketitle
\begin{abstract}We propose a partitioning of the set of unlabelled,
connected cubic graphs into two disjoint subsets named genes
and descendants, where the cardinality of the descendants is much larger
than that of the genes. The key distinction between the two subsets is
the presence of special edge cut sets, called
crackers, in the descendants. We show that every descendant can be created
by starting from a finite set of genes, and introducing the required crackers
by special breeding operations. We prove that it is always possible to identify genes that can be used to generate any given descendant, and provide inverse operations that enable their reconstruction.
A number of interesting properties of genes may be inherited by the descendant, and we
therefore propose a natural algorithm that decomposes a descendant into its ancestor genes. We
conjecture that each descendant can only be generated by starting with a unique set of ancestor genes. The latter is supported by numerical experiments.

\end{abstract}


\section{Introduction}\label{s:Intro}

In this paper, we study the family $S:=S(N)$ of undirected,
unlabelled, connected cubic graphs on $N=2m$ vertices, where $m$
is a positive integer greater or equal to $2$. We recall that a
cubic graph has exactly three edges incident on every vertex. In
our study we will be concerned with the problem of generating
complex cubic graphs by appropriate compositions of simpler cubic graphs
in a manner somewhat analogous to genetic ``breeding". This will
be achieved with the help of six ``breeding operations". Cubic graphs
that cannot be seen as resulting from such operations will be called
``genes", and all other cubic graphs will be called ``descendants".
Of course, descendants constitute a majority of cubic graphs. We will
prove that every descendant graph can be created from a finite set of
genes. Since a lot of the structure in a descendant is inherited
from the genes, many properties of those genes are inherited as well. This
motivates the introduction of \lq\lq inverse operations" that, ultimately, decompose
a given descendant into a set of \lq\lq ancestor genes". Such a
decomposition could subsequently be used with other graph theory algorithms to improve
solving time.

Undirected cubic graphs have been extensively studied in the
literature (e.g., see \cite{hararybook}, \cite{Harary}
\cite{petersenbook}). In particular, there are now programs that enumerate
all instances (up to an automorphism) of cubic graphs on $N$
vertices (e.g., see \cite{genreg}).  A significant line of
research concerns generation of cubic graphs with the help of an
exponential generating function (e.g., see \cite{brinkmann}, \cite{genreg},
\cite{royle}). To the extent that, in this paper, we generate
descendant cubic graphs, the present contribution is conceptually,
but not methodologically, related to the preceding.

The construction of ancestors and descendants of cubic graphs could be
seen as being related to the work of McKay and Royle
\cite{McKay1986}. However, a feature of our approach is that all ancestors and
descendants remain in the class of cubic graphs, and collectively constitute
the entire set of connected cubic graphs. This enables us
to study those properties of cubic graphs that may be ``inherited"
from simpler cubic graphs. Furthermore, the construction methods
are quite different from those in \cite{McKay1986}. The work of Batagelj
\cite{batagelj} is also similar to ours in that it is concerned with the
generation of complex cubic graphs from simpler cubic graphs, where
the former maintain the properties of the latter. In fact, some of the
generating rules given in \cite{batagelj} are analogous to some of the
breeding operations in this paper. However, the two approaches differ
in that Batagelj's approach involves starting with a single cubic graph
and replacing particular structures within the graph sequentially. In our
approach, the majority of development occurs by combining multiple
cubic graphs together, with the resulting graph being more complex than
both parents.

Importantly, preliminary numerical experimentation revealed a surprising \lq\lq unique ancestors property" that is either very common or, as we believe, universal.
Namely, we conjecture that: {\em Given any descendant cubic graph $G_D$ there exists a unique set of genes from which $G_D$ can be derived by a finite sequence of the six breeding operations.} Of course, there may well be a multitude of sequences of breeding operations that lead from the ancestor genes to $G_D$. However, if the above conjecture holds, then all possible decomposition pathways of $G_D$ into its ancestor genes will ultimately lead to a unique set of ancestor genes.

\section{Preliminaries}

 At the conclusion of this manuscript, we conjecture
that any given descendant can be obtained via our approach through a particular
set of ancestor genes, and provide numerical evidence supporting this conjecture.
If true, this conjecture, combined with the proof of existence of ancestor genes for any descendant,
implies that our approach permits a unique decomposition for every cubic graph.

A standard concept in graph theory is that of {\em edge
connectivity}. In the simple case in which removing a single,
specific edge would disconnect the graph, that edge is called a
{\em bridge}. Extending this notion a set of edges $EC$
constitutes an {\em edge cut} if their removal disconnects the
graph. The smallest cardinality of an edge cut in a given graph
$G$ is defined to be the {\em edge-connectivity} of that graph.
Of course, cubic graphs can be $k-$edge-connected for only three
values of $k$: $1,2$ or $3$.

We note, however, that in some $3-$edge-connected cubic graphs
(e.g., the famous Petersen graph, see Figure \ref{fig: Petersen})
the removal of any minimal edge cut
set isolates a single vertex. Arguably, such a partition of
vertices is in a sense degenerate, and prevents a more refined
classification of cubic graphs. To address this problem and
achieve a finer classification we introduce a special class of
edge cut sets that we name {\em crackers}. The latter are defined
as follows.

\begin{definition}An edge cut set of a cubic graph $\Gamma$ consisting of $k$ edges is a
{\em $k$-cracker} if no two edges are adjacent in the sense of
being incident on the same vertex, and no proper subset of these
edges disconnects the graph.\label{k-cracker}\end{definition}

\begin{lemma}The removal of a cracker from a connected graph results in exactly two disjoint components.\end{lemma}

\begin{proof}Since any cracker is an edge cut set, the removal of a cracker must result in at least two disjoint components. Assume there exists a cracker $C = \{e_1, \hdots, e_k\}$ whose removal results in more than two disjoint components. We can think of the removal of $C$ as a series of individual edge removals, for each edge in $C$. It is clear that the removal of a single edge cannot result in more than one additional disjoint component. Therefore, after removing edges $e_1, \hdots, e_{k-1}$, but before removing edge $e_k$, there must already be at least two disjoint components. However, this implies that a proper subset of $C$ also disconnects the graph, which violates Definition \ref{k-cracker}. Therefore, $C$ is not a cracker, and the initial assumption is false.\end{proof}

Of course, any cracker is not only an edge cut set, but in fact
a {\em cyclic edge cut set}. This is clear because its removal disconnects the graph into two connected
subgraphs, each of which must contain at least three vertices. Since each remaining vertex has degree at least
two (as only non-adjacent edges were removed), these connected subgraphs must contain cycles. It is then clear that a minimal cyclic edge cut set in a given cubic graph is a cracker of minimal size for that graph. However, larger cyclic edge cut sets may contain adjacent edges, and therefore not all cyclic edge cut sets are crackers.

We recall that the {\em girth}, $g$, of a graph is the length of the
shortest (nontrivial) cycle in the graph. In many cases, a
nontrivial cycle determining a cubic graph's girth automatically
defines a $g$-cracker made up of edges that are not in the cycle but
which have one vertex on the cycle. However, many graphs have
crackers of size less than $g$. For instance, it is easy to see
that in the graph given in Figure \ref{fig: twoCracker}, $g=3$ but
the edges $e_1$ and $e_2$ form a $2$-cracker and no $1$-cracker
exists in this graph.

\begin{figure}[h]\begin{center}\epsfig{scale=0.54, figure=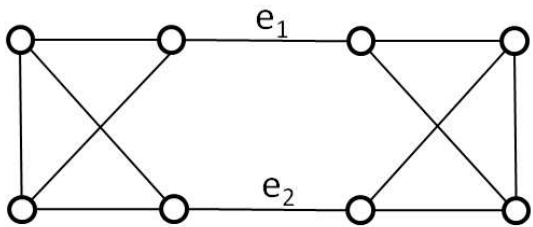}\\
\caption{An 8-vertex cubic graph with girth $3$ and the two cracker
$\{e_1, e_2\}$}\label{fig: twoCracker}\end{center}\end{figure}

We note also that there are only two connected cubic graphs, on $4$
and $6$ vertices respectively, that contain no crackers at all
(see Figure \ref{fig-graph_4_and_graph_6} and the discussion in Section \ref{s:Motiv}).

\vspace*{0.5cm}\begin{remark}Note that, since the minimal cyclic edge cut set in a
cubic graph is a cracker, it is clear that for any given cubic graph,
the cyclic edge connectivity is equal to the size of the smallest cracker in that graph.
For the sake of simplifying the notation, we will refer to a cyclically
$k$-connected graph as a C$k$-connected graph. The class of all C$k$-connected
cubic graphs on $N$ vertices will be denoted by $S_{k}(N)$, or simply
by $S_{k}$, when the number of vertices is fixed.\label{rem-cyclic}\end{remark}
%
%

We note that the famous Petersen graph is C5-connected in the
above sense (see Figure \ref{fig: Petersen}), as the edges
connecting the ``inner-star" to the outer boundary form one of a
number of $5-$crackers, and no smaller crackers exist in this
graph.
\begin{figure}[h]\begin{center}\epsfig{scale=0.41, figure=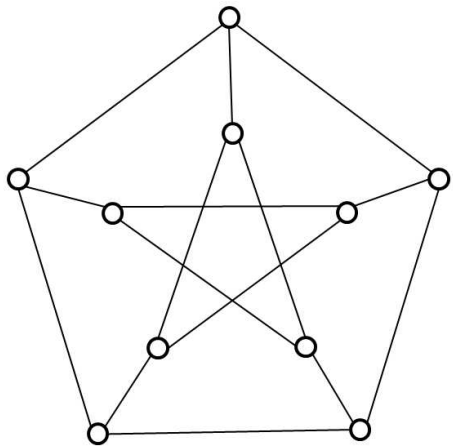}\\
\caption{Petersen graph which is C5-connected}\label{fig:
Petersen}\end{center}\end{figure}

\section{Motivation}\label{s:Motiv}

It follows immediately from Remark \ref{rem-cyclic} that (for
fixed $N \geq 8$) the class $C$ of connected cubic graphs can be partitioned as
\begin{equation*}\label{Partition}
S= \bigcup_{i=1}^M S_i,
\end{equation*}
where $M \leq \displaystyle\frac{N}{2}$. The choice of upper bound is conservative because at most $\displaystyle\frac{N}{2}$ non-adjacent edges can be chosen in any graph of size $N$, but in reality
it is likely that far fewer than $\displaystyle\frac{N}{2}$ partitions will be required for any given $N$.
For instance, when $N = 20$, $M = 6$.

\begin{definition}If a cubic graph $\Gamma$ is C$k$-connected
for $k \geq 4 $, we call it a {\em gene}. Otherwise, we call
$\Gamma$ a {\em
descendant}.\label{def-gene_descendant}\end{definition}

The reasoning behind the choice of names {\em gene} and {\em
descendant} is made clear in Section \ref{sec: breeding},
where we demonstrate that any descendant can be obtained from a
set of genes, through the use of prescribed {\em breeding operations} that
introduce crackers into a descendant. Experiments have shown that
genes are far less numerous than descendants.

It was mentioned earlier that two cubic graphs, namely the 4-vertex gene $\Gamma_4^*$ and the 6-vertex gene $\Gamma_6^*$, contain no crackers at all. These two graphs can be seen in Figure \ref{fig-graph_4_and_graph_6}. The following lemma proves that every other cubic graph contains at least one cracker, and that the size of the smallest cracker is bounded above by the girth of the graph.

\begin{figure}[h!]\begin{center}\epsfig{scale=0.5, figure=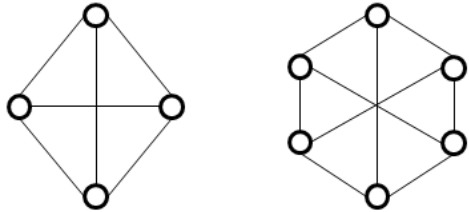}\\
\caption{Cubic graphs $\Gamma_4^*$ and
$\Gamma_6^*$, which contain no crackers}\label{fig-graph_4_and_graph_6}\end{center}\end{figure}

\begin{lemma} \label{prop-boundgirth} Except for $\Gamma_4^*$ and $\Gamma_6^*$, all connected cubic graphs contain at least one cracker of size no more than the girth $g$ of the graph.\label{lem-all_cracker}
\end{lemma}

\begin{proof}The cases of $\Gamma_4^*$ and $\Gamma_6^*$ can be confirmed by inspection. There is one other cubic graph containing 6 vertices, with girth 3, which contains a 3-cracker, as displayed in Figure \ref{fig-envelope}. So the lemma is true for $N < 8$.

For any cubic graph containing 8 or more vertices, it was proved in Lou et al \cite{lou} that there exists at least one cyclic edge cut set of size $g$. The smallest cyclic edge cut set in a cubic graph is a cracker, so it is clear that
the smallest cracker can be of size no bigger than $g$.\end{proof}

\begin{figure}[h!]\begin{center}\epsfig{scale=0.5, figure=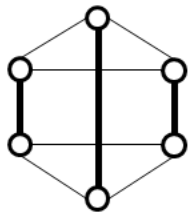}\\
\caption{A 6-vertex cubic graph, with girth 3, which contains a 3-cracker comprising the highlighted edges}\label{fig-envelope}\end{center}\end{figure}

In Section \ref{sec-results}, it is proved that any descendant graph can be
obtained from a set of genes. It is our hope that descendants inherit
many of their properties from the genes used to construct them. If so, any
analysis of such a descendant could be reduced to the problem of analysing
the component genes, which are often much smaller than the descendant.
The subsequent investigation of graph theoretic properties in sets of genes
is a natural topic for future research.

One such graph theoretic property of interest is
that of {\em Hamiltonicity}, that is, the property of containing
a simple cycle of length equal to the number of vertices in the graph.
We observe that non-Hamiltonian genes are extremely rare. Even excluding
the (trivially non-Hamiltonian) {\em bridge graphs}, non-Hamiltonian descendants
constitute a large majority of the remaining non-Hamiltonian graphs; see Table
\ref{tab:prevalanceNH}. The second column of that table, labelled
by $\mbox{NH}_1$, lists the percentages of bridge graphs relative to the
total cardinality of non-Hamiltonian graphs\footnote{For a more complete study of
the prevalence of cubic bridge graphs relative to the total set of cubic non-Hamiltonian
graphs, see Filar et al \cite{conjecturepaper}.}, denoted by $\mbox{NH}$. The third
column labelled by $\mbox{NH}_{2+} := \mbox{NH} \setminus \mbox{NH}_1$,
lists the percentages of graphs two or more cyclically edge connected
relative to the cardinality of $\mbox{NH}$. The fourth column
labelled by $\mbox{NH}_{4+} $, lists the percentages of graphs
four or more cyclically edge connected relative to the cardinality of
$\mbox{NH}$. Finally, the fifth column labelled by
$\mbox{NH}_{4+}/\mbox{NH}_{2+} $, lists the percentages of graphs
four or more cyclically edge connected relative to the cardinality of
all non-bridge, non-Hamiltonian, graphs in $\mbox{NH}_{2+}$. We
shall define graphs in $\mbox{NH}_{4+}$ as {\em mutants}, a name
that properly reflects their exceptionality.  For instance, we note
from the fourth column of Table \ref{tab:prevalanceNH} that with
$N=18$ only $0.12$ of one percent are mutants, which corresponds to
two (out of 1666) non-Hamiltonian cubic graphs on $18$ vertices.
These two mutants are the famous Blanu\u{s}a Snarks \cite{blanusa}.

In fact, it is no coincidence that the Blanu\u{s}a Snarks appear as mutants
in this framework. In Read and Wilson \cite{readwilson}, the
definition of an irreducible Snark is given as a cubic graph
with edge chromatic number of 4, girth 5 or more, and not containing
three edges whose deletion results in a disconnected graph, each of
whose components is nontrivial. This final condition, along with the
well known fact that all Snarks are non-Hamiltonian, is akin to
our definition of a mutant. Therefore, the set of all mutants is a superset
of the set of all irreducible Snarks. However, some non-Snark mutants do exist,
and therefore have either girth 4 or an edge chromatic number of 3 (or both).
In particular, the BH-Mutant, displayed in Figure \ref{fig: BHmutant} is the smallest
non-Snark mutant, and has an edge chromatic number of 3. There are 16 further non-Snark
mutants of size 22, one of which is the Zircon-Mutant, also displayed in
Figure \ref{fig: BHmutant}. The Zircon-Mutant also has an edge chromatic
number of 3.

\begin{table}[h!]
\begin{center}
\begin{tabular}{|c|c|c|c|c|}
  \hline
            & $\mbox{NH}_1$     & $\mbox{NH}_{2+}$    & NH$_{4+}$ &$\mbox{NH}_{4+}/\mbox{NH}_{2+}$ \\ \hline
  $10$ Vertex & $50.00\%$ & $50.00\%$ & $50.00\%$ &$100\%$\\ \hline
  $12$ Vertex & $80.00\%$ & $20.00\%$ & $0\%$ &$0\%$\\ \hline
  $14$ Vertex & $82.86\%$ & $17.14\%$ & $0\%$ &$0\%$\\ \hline
  $16$ Vertex & $84.93\%$ & $15.07\%$ & $0\%$ &$0\%$\\ \hline
  $18$ Vertex & $86.13\%$ & $13.86\%$ & $0.12\%$ &$0.86\%$\\ \hline
  $20$ Vertex & $87.40\%$ & $12.60\%$ & $0.05\%$ &$0.38\%$\\ \hline
  $22$ Vertex & $88.59\%$ & $11.41\%$ & $0.02\%$ &$0.21\%$\\ \hline
\end{tabular}
\caption{Distribution of Non-Hamiltonian (NH) C$k$-connected graphs} \label{tab:prevalanceNH}
\end{center}
\end{table}
\begin{figure}[h!]\begin{center}\epsfig{scale=0.54, figure=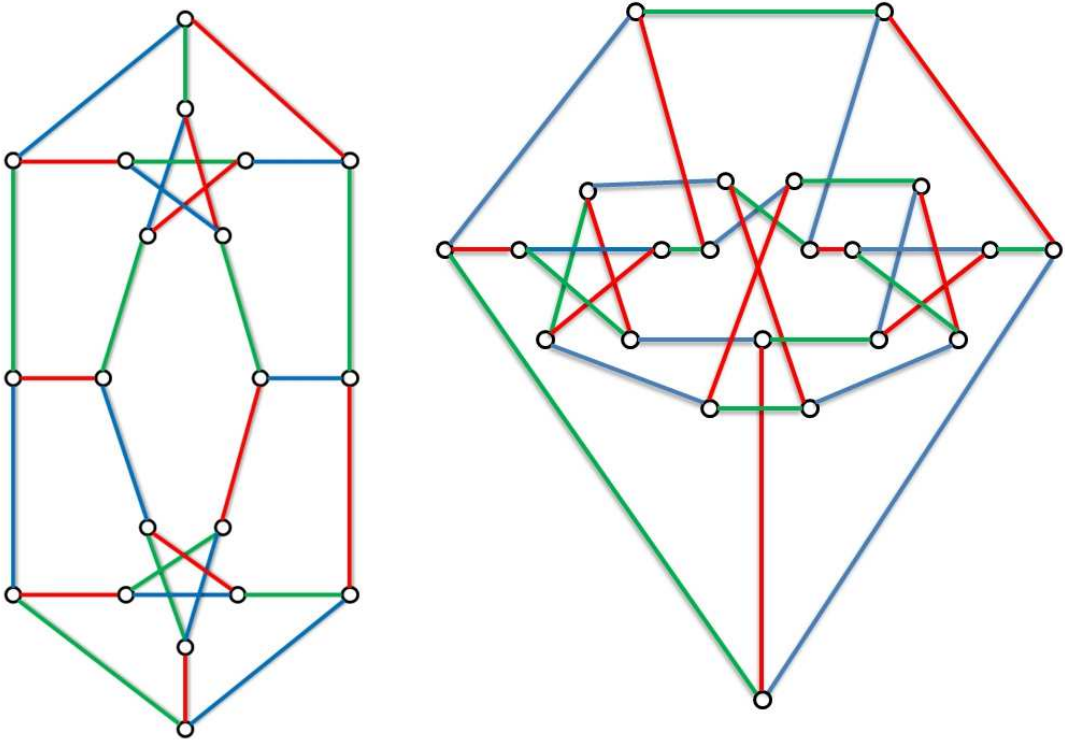}\\
\caption{BH-Mutant which is a girth 5, C4-connected non-Hamiltonian
non-Snark mutant containing 20 vertices, and Zircon-Mutant which is a girth 5,
C5-connected non-Hamiltonian non-Snark mutant containing 22 vertices.
The colours identify Tait Colorings}\label{fig:
BHmutant}\end{center}\end{figure}

\section{Breeding and parthenogenic operations}\label{sec: breeding}

From Definition \ref{def-gene_descendant}, it is clear that genes do not contain
any 1-crackers, $2$-crackers, or $3$-crackers. Collectively, we
refer to 1-crackers, $2$-crackers and $3$-crackers as {\em cubic
crackers}. As a corollary, descendants must contain at least one cubic
cracker. It then seems plausible that we might be able to construct any given
descendant by combining two or more cubic graphs together in such
a fashion as to create the cubic crackers present in that
descendant graph. Since there are three different types of cubic
crackers, we define three {\em breeding operations} that map two
cubic graphs to a single descendant by inserting a cubic cracker
between them in such a fashion as to retain cubicity. In such a
case, we say that the descendant has been obtained by {\em
breeding}. We refer to the original two cubic graphs as the {\em
parents} of the descendant graph, and likewise the descendant
graph is the {\em child} of the two parents.

Note that the following operations are defined only for cubic graphs. Although
they work for disconnected cubic graphs, in this manuscript we are interested
only in connected cubic graphs, and make the assumption that all input graphs are indeed connected and cubic.

\subsection{Breeding operations}

\begin{definition}
A {\em type 1 breeding operation} is a function $\mathcal{B}_1$
defined on the tuple $(\Gamma_1,\Gamma_2,e_1,e_2)$, where
$\Gamma_1 = \left<V_1,E_1\right>$ and $\Gamma_2 =
\left<V_2,E_2\right> $ are cubic graphs, and furthermore, $e_1 =
(a,b) \in E_1$ and $e_2 = (c,d) \in E_2$. This function maps such
a tuple onto another tuple $(\Gamma_D, e)$ as follows
\begin{eqnarray}
\mathcal{B}_1(\Gamma_1,\Gamma_2,e_1,e_2) & = & (\Gamma_D,
e),\nonumber\end{eqnarray}

where $\Gamma_D = \left<V_D,E_D\right>$ and $ \{ e = (v_1, v_2) \}
\in E_D$. The new set of vertices is $V_D = V_1\; \cup \;V_2
\;\cup \;v_1 \;\cup \;v_2$. The new set of edges is $E_D = (E_1
\backslash e_1) \;\cup\; (E_2 \backslash e_2) \;\cup\;
\{(a,v_1),(b,v_1),(c,v_2),(d,v_2),(v_1,v_2)\}$.\label{def-breed1}\end{definition}

\begin{figure}[h]\begin{center}
\hspace*{0cm}\epsfig{scale=0.35, figure=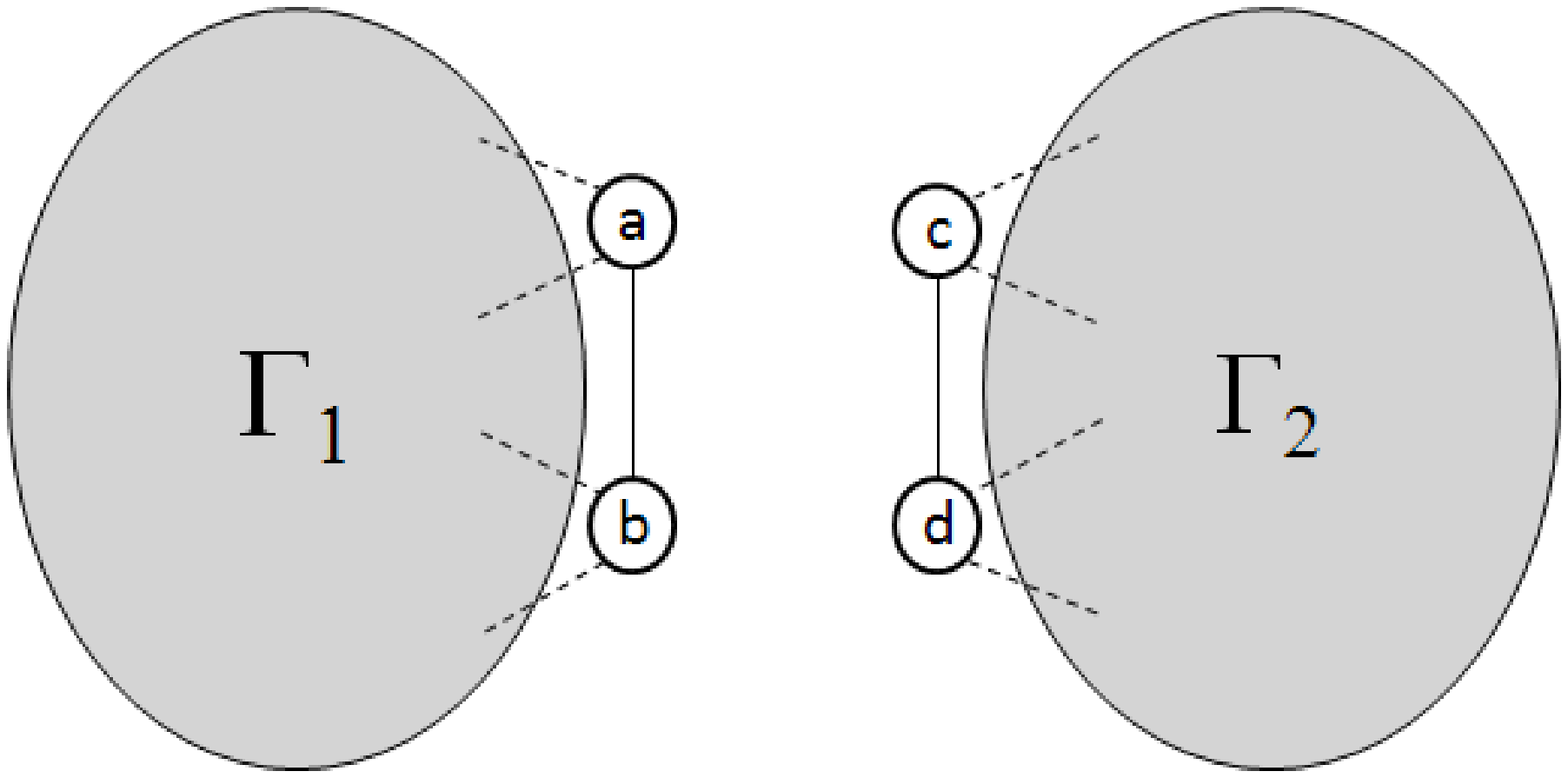}\hspace*{1.0cm}\epsfig{scale=0.35, figure=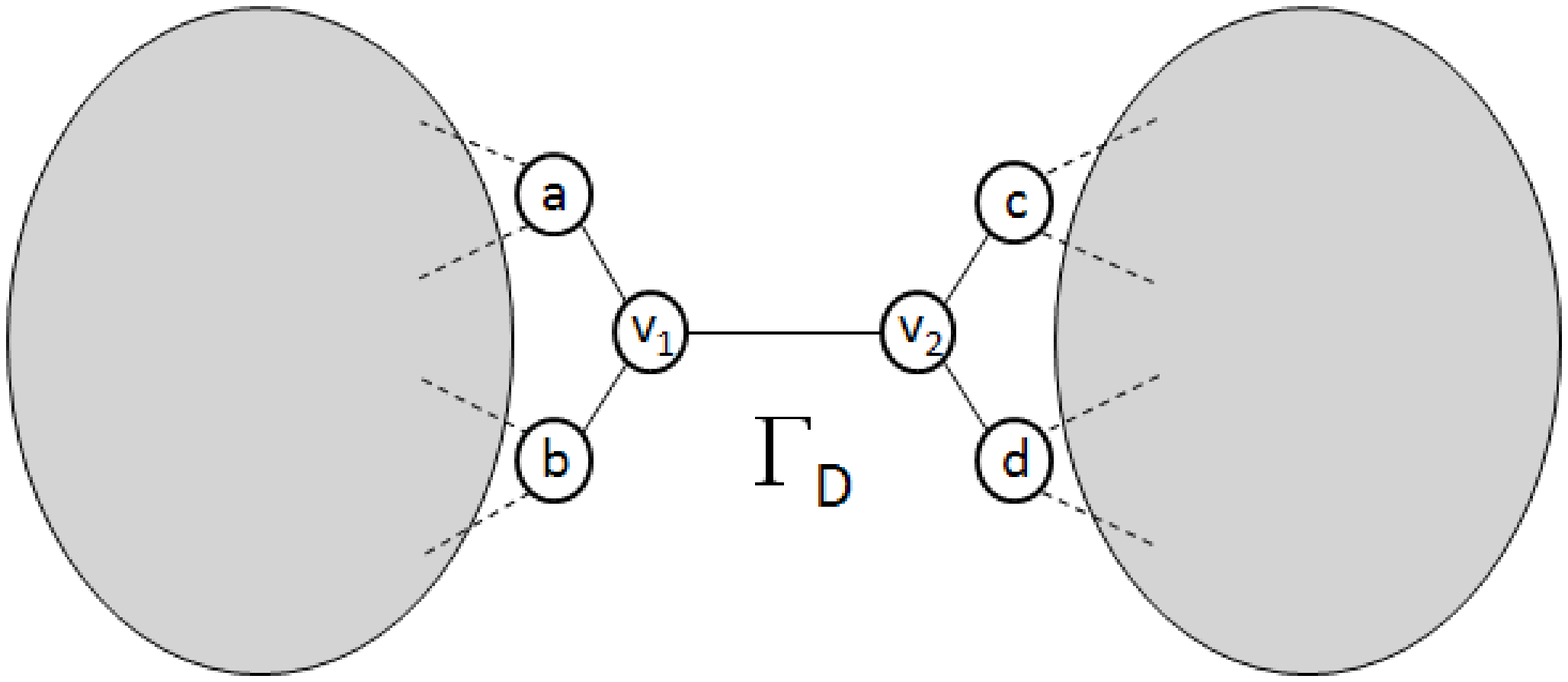}\\
\caption{Graphs $\Gamma_1$, $\Gamma_2$ and $\Gamma_D$ as described
in Definition \ref{def-breed1}}\label{fig-type_1_breeding}
\end{center}\end{figure}

Note that a type 1 breeding operation always outputs a {\em bridge
graph} (that is, a C1-connected graph). See Figure
\ref{fig-type_1_breeding} for an illustration.

\begin{definition}
A {\em type 2 breeding operation} is a function $\mathcal{B}_2$
defined on the tuple $(\Gamma_1,\Gamma_2,e_1,e_2)$, where
$\Gamma_1 = \left<V_1,E_1\right>$ and $\Gamma_2 =
\left<V_2,E_2\right> $ are cubic graphs, and furthermore, $e_1 =
(a,b) \in E_1$ and $e_2 = (c,d) \in E_2$ and neither edge is a
$1$-cracker. This function maps such a tuple onto another tuple
$(\Gamma_D, e_3, e_4)$ as follows
\begin{eqnarray}\mathcal{B}_2(\Gamma_1,\Gamma_2,e_1,e_2) & = & (\Gamma_D, e_3, e_4),\nonumber\end{eqnarray}

where $\Gamma_D = \left<V_D,E_D\right>$ and $ \{ e_3 = (a,c) , e_4
= (b,d) \} \in E_D$. The new set of vertices is $V_D = V_1\;
\cup \;V_2$. The new set of edges is $E_D = (E_1 \backslash e_1)
\;\cup\; (E_2 \backslash e_2) \;\cup\;
\{(a,c),(b,d)\}$.\label{def-breed2}\end{definition}

\begin{figure}
[h]\begin{center}\hspace*{0cm}\epsfig{scale=0.35, figure=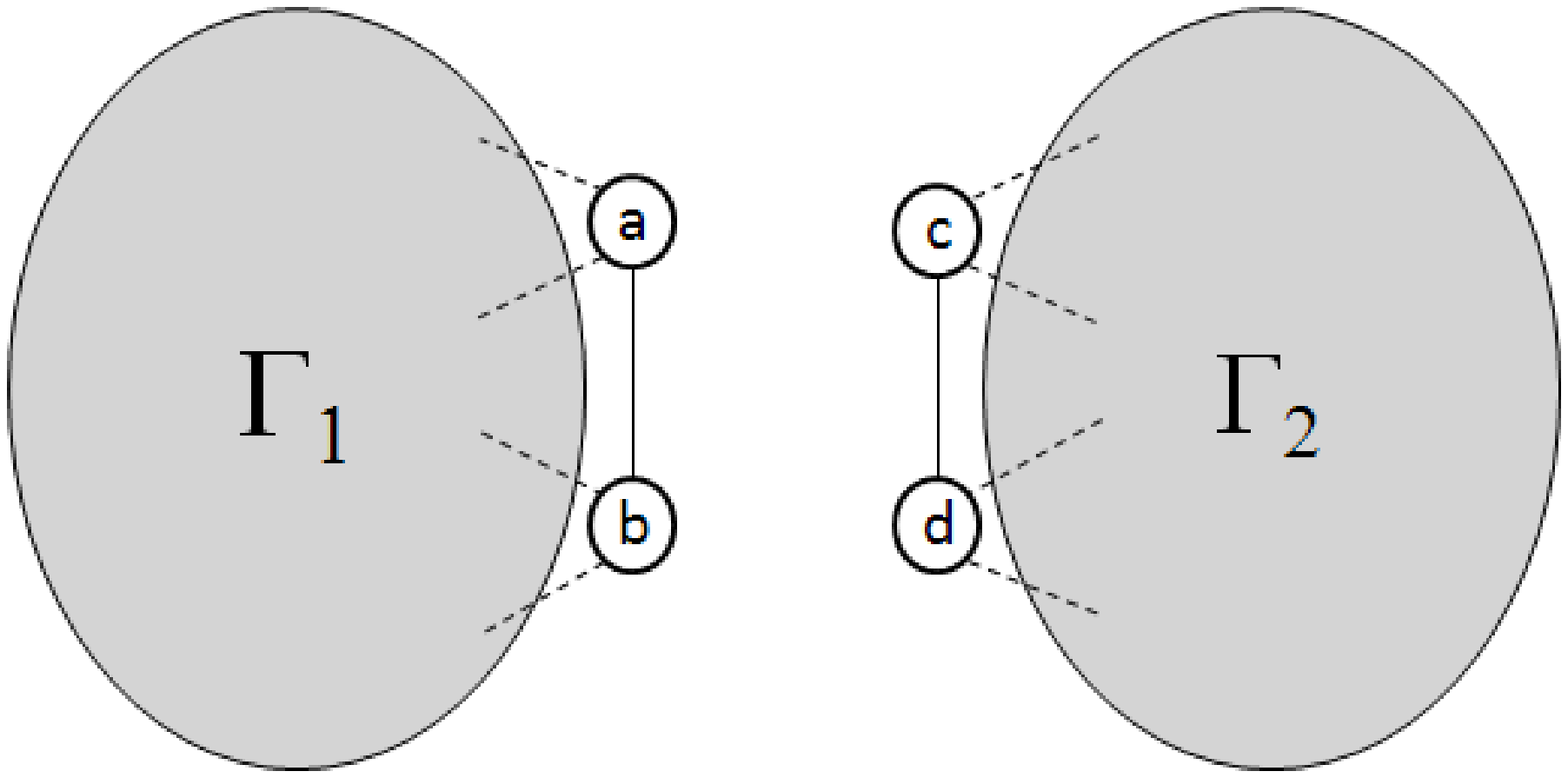}\hspace*{0.8cm}\epsfig{scale=0.35, figure=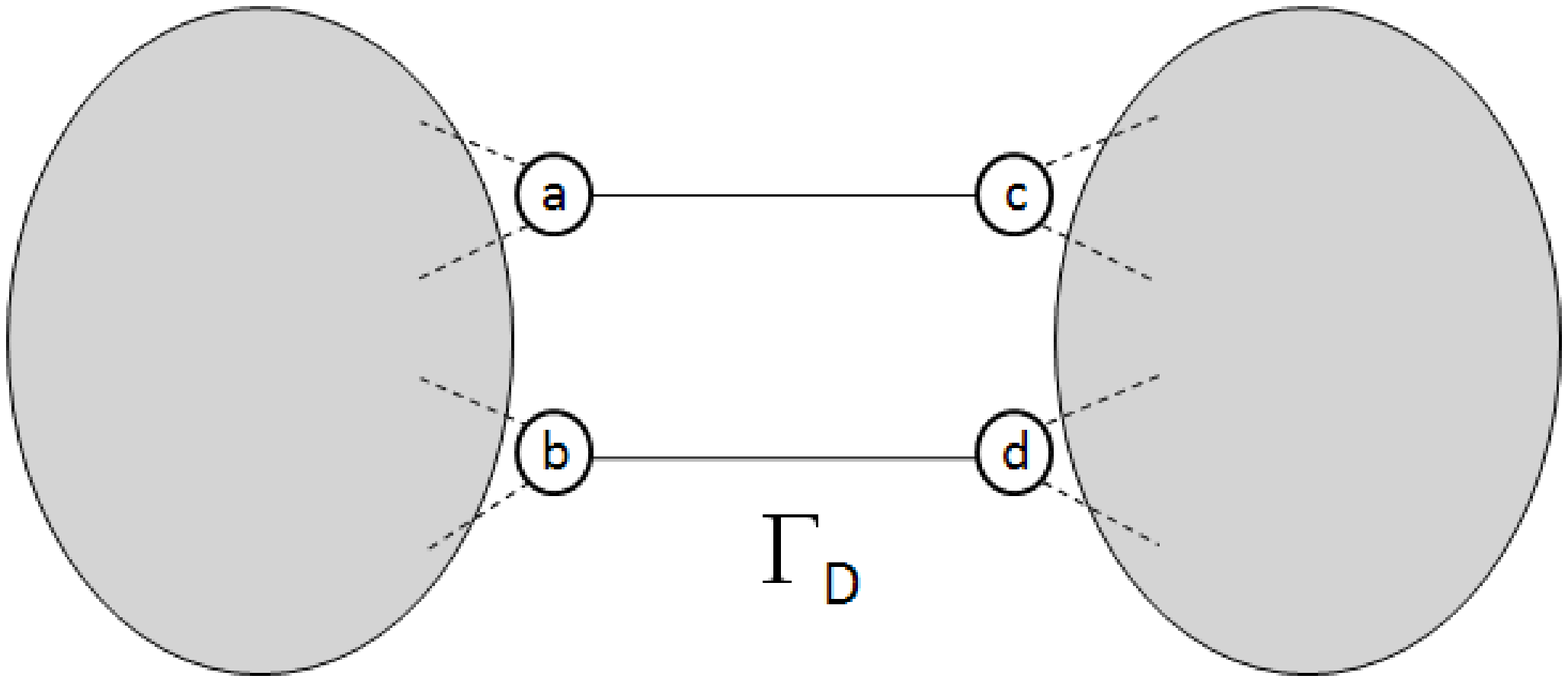}\\
\caption{Graphs $\Gamma_1$, $\Gamma_2$ and $\Gamma_D$ as described
in Definition
\ref{def-breed2}}\label{fig-type_2_breeding}\end{center}
\end{figure}

Clearly $\Gamma_D$ contains the $2$-cracker $\{(a,c),(b,d)\}$. See
Figure \ref{fig-type_2_breeding} for an illustration. Note also
that a type 2 breeding operation always creates a 2-edge-connected
descendant, unless either of $\Gamma_1$ or $\Gamma_2$ is
1-edge-connected (in which case, $\Gamma_D$ is also
1-edge-connected).

\begin{definition}
A {\em type 3 breeding operation} is a function $\mathcal{B}_3$
defined on the tuple $(\Gamma_1,\Gamma_2,v_1,v_2,a,b,\\
c,d,e,f)$, where $\Gamma_1 = \left<V_1,E_1\right>$ and $\Gamma_2 =
\left<V_2,E_2\right> $ are cubic graphs, and furthermore, $v_1 \in
V_1$ is incident to vertices $a$ , $b$ and $c \in V_1$ and $v_2
\in V_2$ is incident to vertices $d$, $e$ and $f \in V_2$. None
of the edges adjacent to $v_1$ or $v_2$ are $1$-crackers. This
function maps such a tuple onto another tuple $(\Gamma_D, e_1,
e_2, e_3)$ as follows
\begin{eqnarray}\mathcal{B}_3(\Gamma_1,\Gamma_2,v_1,v_2,a,b,c,d,e,f) & = & (\Gamma_D, e_1,
e_2, e_3),\nonumber\end{eqnarray}

where $\Gamma_D = \left<V_D,E_D\right>$, and also, $ \{e_1 = (a,d) ,e_2=(b,e) , e_3 = (c,f)\} \in E_D$). The new set of vertices is $V_D = (V_1 \backslash v_1) \;\cup\; (V_2 \backslash v_2)$. The new set of edges is $E_D = (E_1 \backslash \{(v_1,a),(v_1,b),(v_1,c)\}) \;\cup\; (E_2 \backslash \{(v_2,d),(v_2,e),(v_2,f)\}) \cup\; \{(a,d),(b,e),(c,f)\}$. \label{def-breed3}
\end{definition}

\begin{figure}[h]\begin{center}\hspace*{0cm}\epsfig{scale=0.35, figure=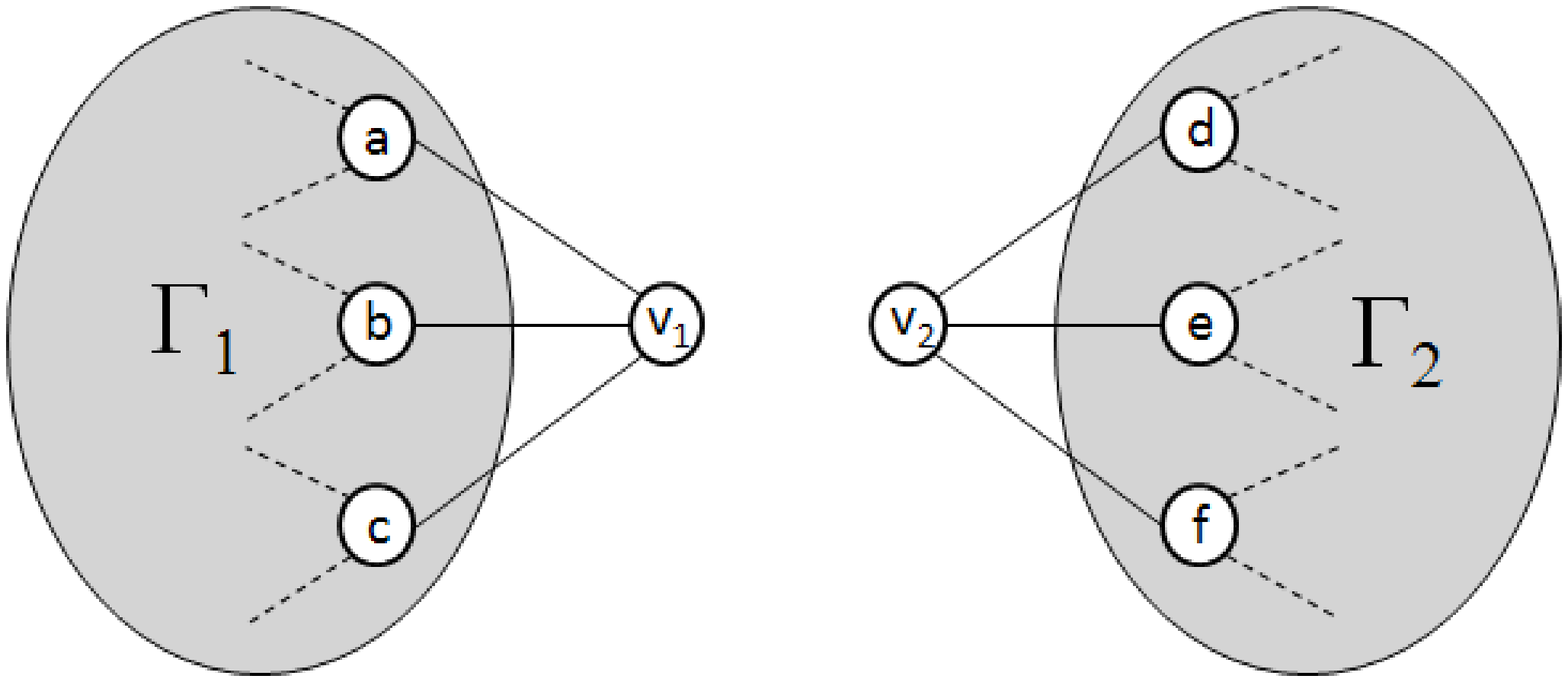}\hspace*{1.0cm}\epsfig{scale=0.35, figure=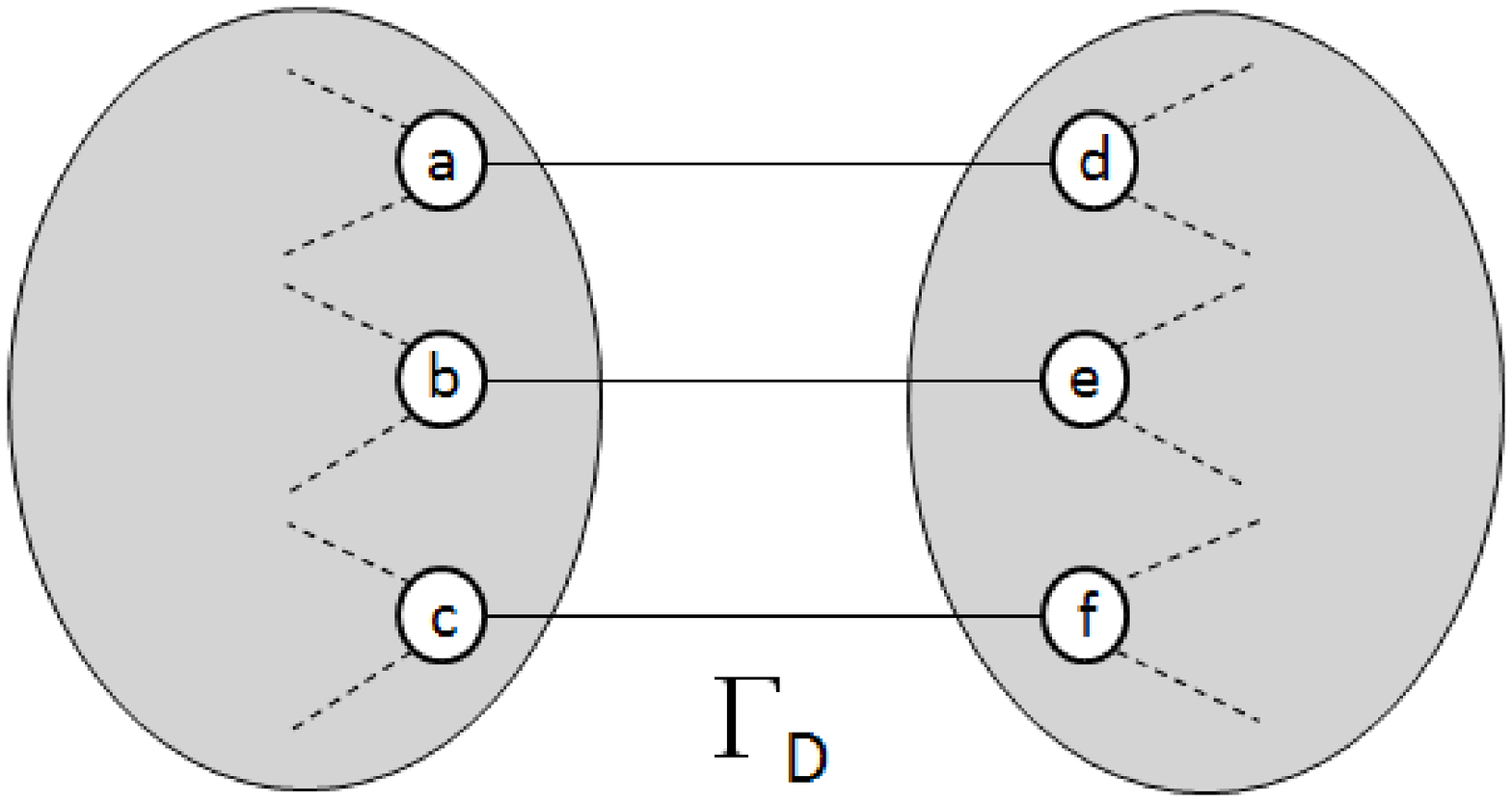}\\
\caption{Graphs $\Gamma_1$, $\Gamma_2$ and $\Gamma_D$ as described
in Definition
\ref{def-breed3}}\label{fig-type_3_breeding}\end{center}\end{figure}

See Figure \ref{fig-type_3_breeding} for an illustration of type 3
breeding.

\subsection{Parthenogenic operations}

In addition to the preceding three breeding operations, we also define three
{\em parthenogenic operations}. These are operations that map a
single descendant to a new, more complex, descendant by replacing
a cracker in the original descendant with two new crackers. We say
that such a new descendant has been obtained from {\em
parthenogenesis}. For simplicity of terminology, we again refer to
the original descendant as the {\em parent} of the new descendant,
and likewise we refer to the new descendant as the {\em child} of
the original descendant. Also for simplicity of terminology, we refer
to the the three breeding operations and the three parthenogenic
operations collectively as the six {\em breeding operations}.

\begin{definition}
A {\em type 1 parthenogenic operation} is a function
$\mathcal{P}_1$ defined on the tuple $(\Gamma_1,e_1)$ where
$\Gamma_1 = \left<V_1,E_1\right>$ is a bridge graph and $e_1 =
(a,b) \in E_1$ is a 1-cracker. This function maps such a tuple
onto another tuple $(\Gamma_D, v_1, v_4)$ as follows
\begin{eqnarray}\mathcal{P}_1(\Gamma_1,e_1) & = & (\Gamma_D, v_1, v_4),\nonumber\end{eqnarray}

where $\Gamma_D = \left<V_D,E_D\right>$ and $\{v_1 , v_4\} \in V_D$.
The new set of vertices is $V_D = V_1 \;\cup\;
\{v_1,v_2,v_3,v_4\}$. The new set of edges is $E_D = (E_1
\backslash e_1) \;\cup\;
\{(a,v_1),(v_1,v_2),(v_1,v_3),(v_2,v_3),(v_2,v_4),(v_3,v_4),\\(v_4,b)\}$. This process inserts an additional
$1$-cracker into $\Gamma_D$.
\label{def-parth1}\end{definition}

\begin{figure}[h]\begin{center}\epsfig{scale=0.35, figure=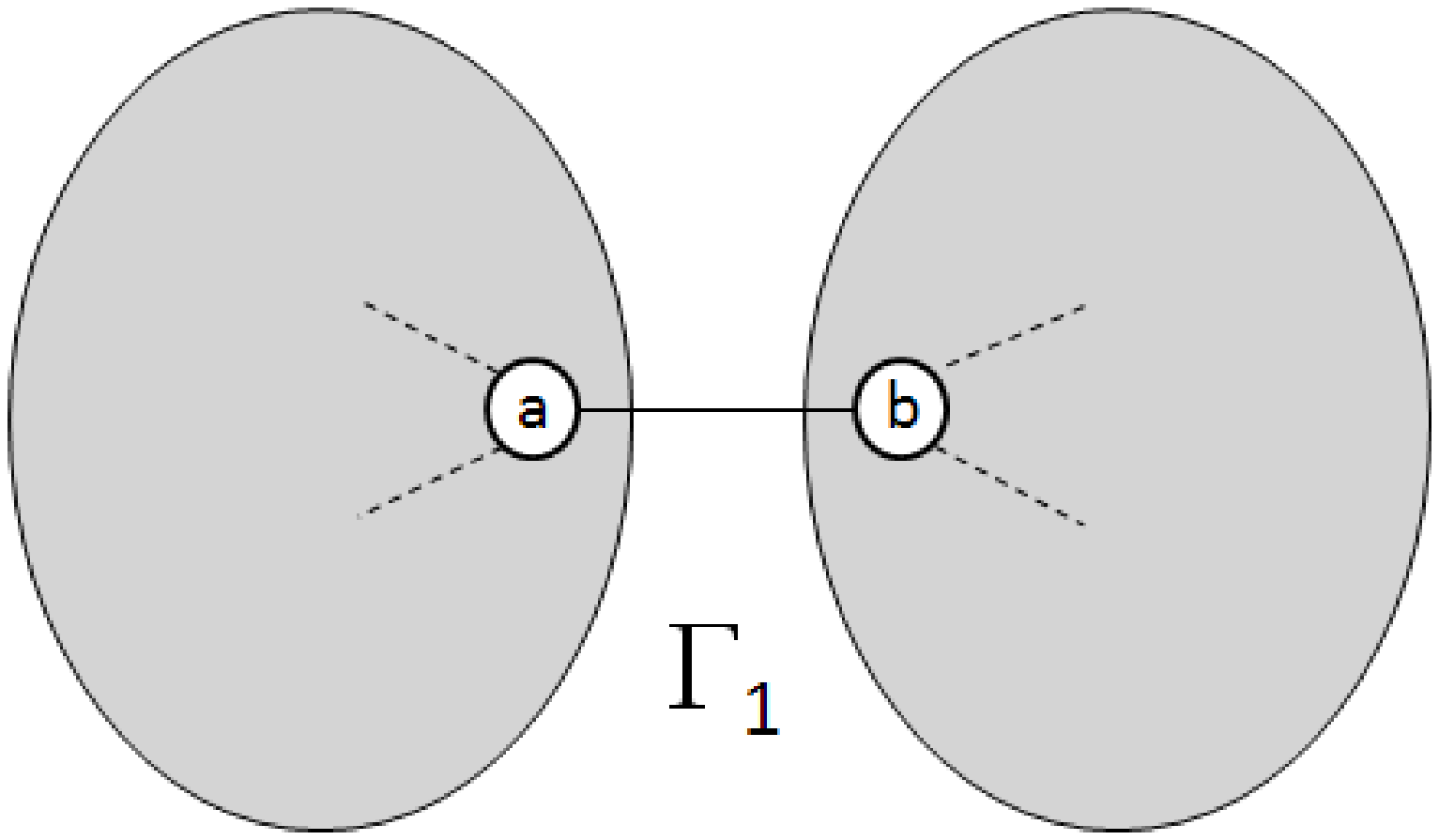}\hspace*{1.5cm}\epsfig{scale=0.35, figure=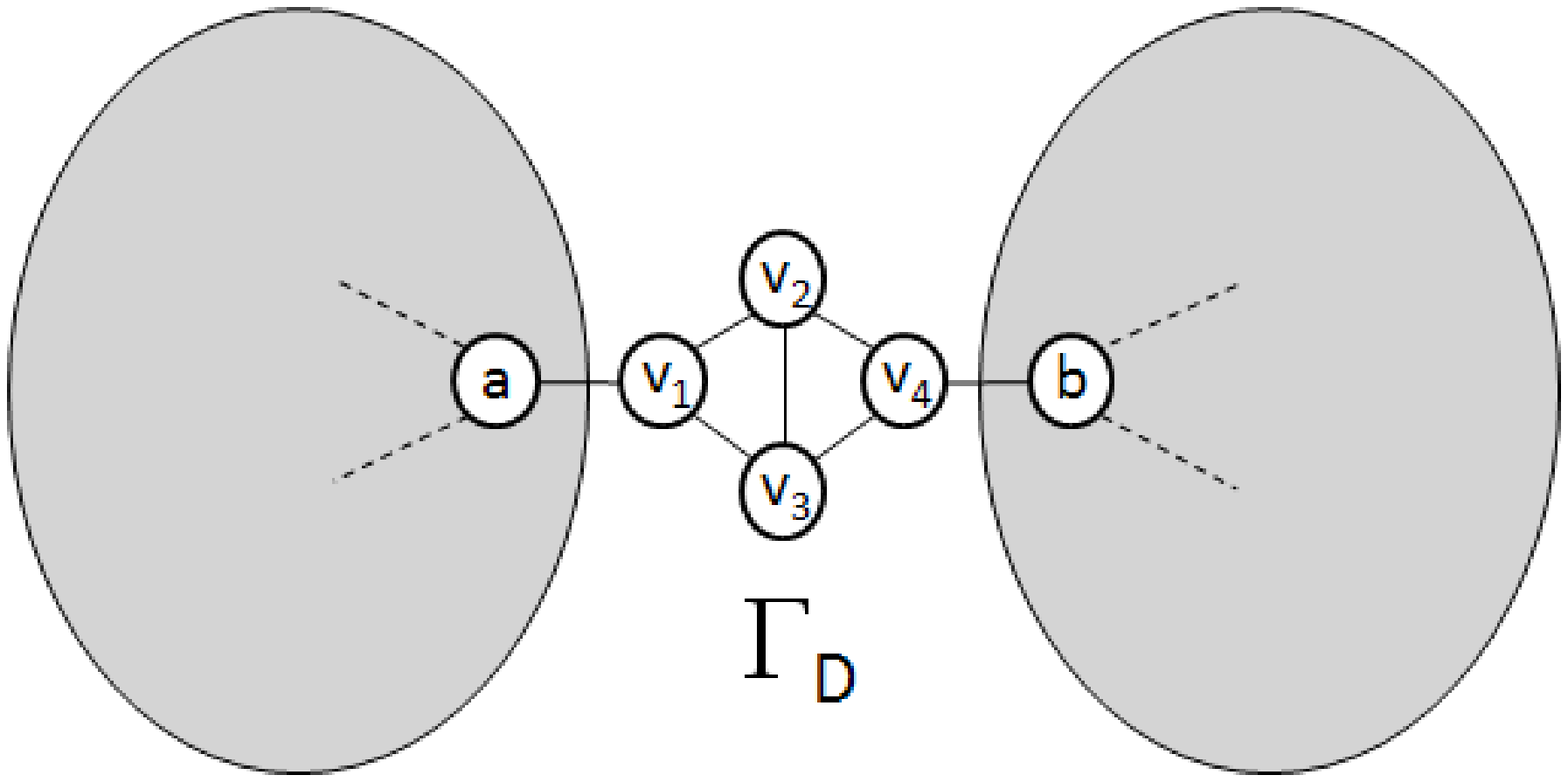}\\
\caption{Graphs $\Gamma_1$ and $\Gamma_D$ as described in
Definition
\ref{def-parth1}}\label{fig-type_1_parthenogenisis}\end{center}\end{figure}

We refer to the subgraph $\Gamma_S =
\left<\{v_1,v_2,v_3,v_4\},\{(v_1,v_2),(v_1,v_3),(v_2,v_3),(v_2,v_4),(v_3,v_4)\}\right>$ as the {\em parthenogenic diamond}, and say
that a type 1 parthenogenic operation inserts a parthenogenic
diamond into a bridge. See Figure \ref{fig-type_1_parthenogenisis}
for an illustration.

\begin{figure}[h]\begin{center}\epsfig{scale=0.30, figure=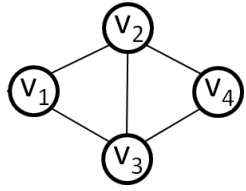}\\
\caption{A parthenogenic
diamond}\label{fig-parthenogenic_diamond}\end{center}\end{figure}

\begin{definition}
A {\em type 2 parthenogenic operation} is a function
$\mathcal{P}_2$ defined on the tuple $(\Gamma_1,e_1,e_2)$ where
$\Gamma_1 = \left<V_1,E_1\right>$ is a cubic graph containing a
$2$-cracker comprising two edges $e_1 = (a,b)$ and $e_2 = (c,d)$.
This function maps such a tuple onto another tuple $(\Gamma_D,
v_1, v_2)$ as follows

\begin{eqnarray}\mathcal{P}_2(\Gamma_1,e_1,e_2) & = & (\Gamma_D,
v_1, v_2),\nonumber\end{eqnarray}

where $\Gamma_D = \left<V_D,E_D\right>$ and $\{ v_1, v_2 \} \in
V_D$. The new set of vertices is $V_D = V_1 \;\cup\; \{v_1,v_2\}$.
The new set of edges is $E_D = (E_1 \backslash \{e_1,e_2\})
\;\cup\; \{(a,v_1),(b,v_1),(c,v_2),(d,v_2),(v_1,v_2)\}$. This
process inserts an additional $2$-cracker into
$\Gamma_D$.\label{def-parth2}\end{definition}

\begin{figure}[h]\begin{center}\hspace*{0cm}\epsfig{scale=0.35, figure=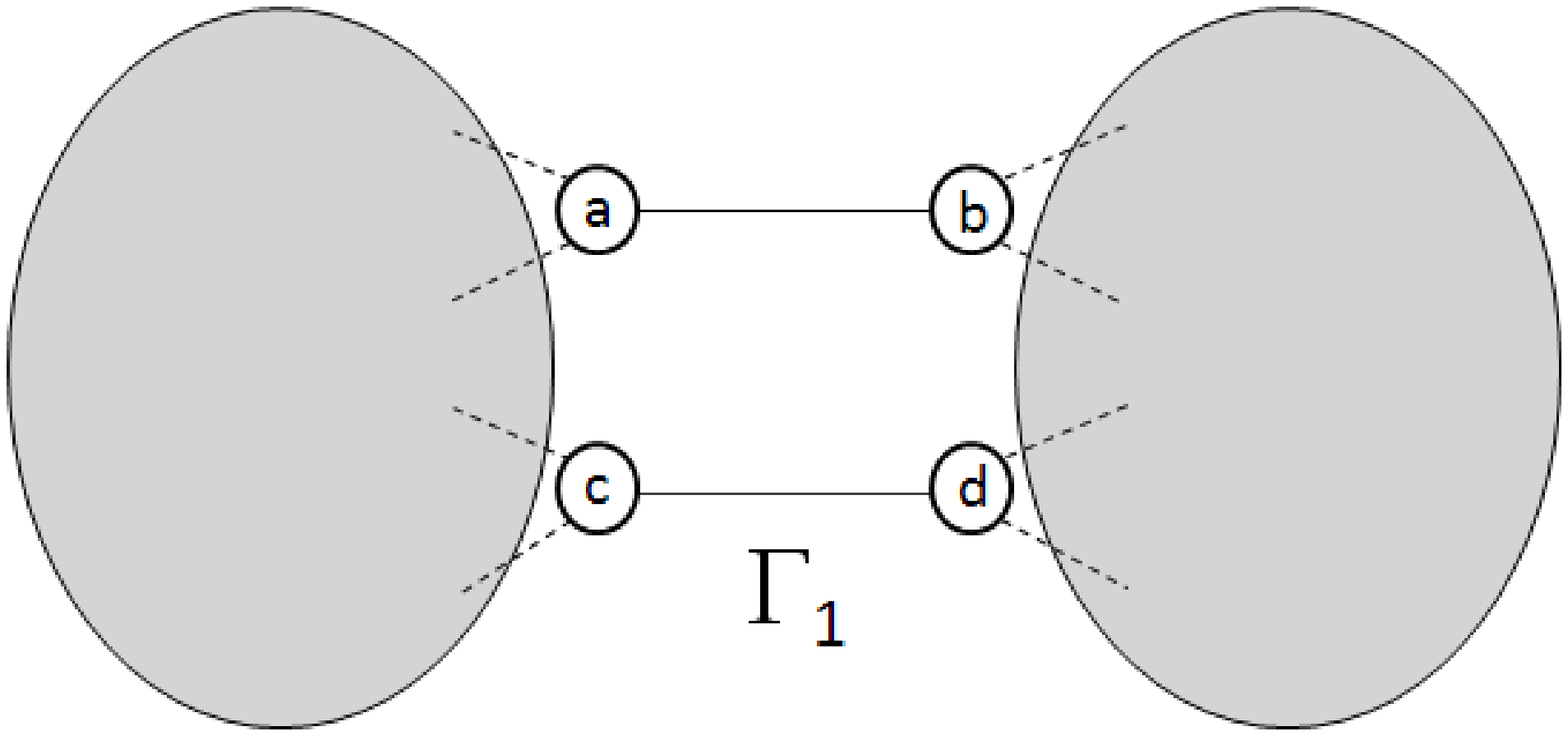}\hspace*{0.7cm}\epsfig{scale=0.35, figure=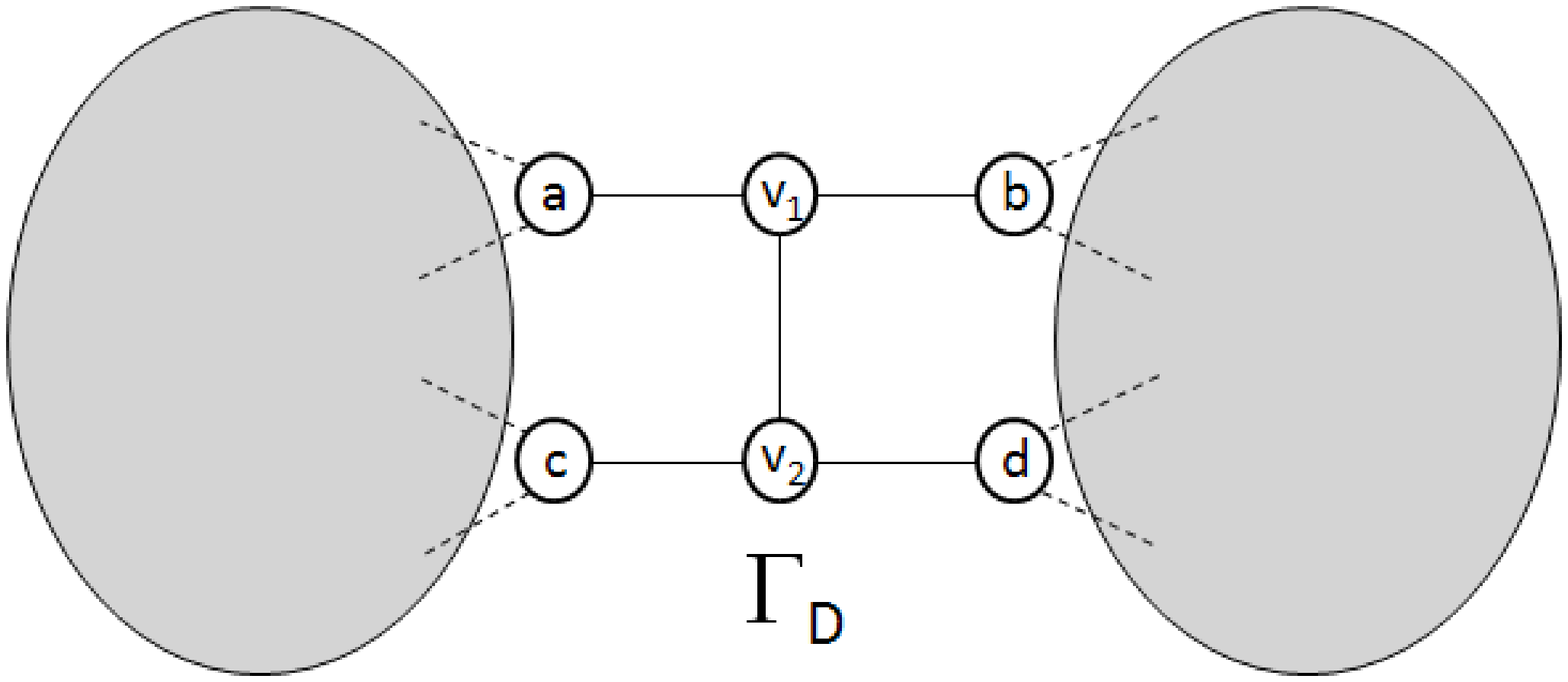}\\
\caption{Graphs $\Gamma_1$ and $\Gamma_D$ as described in
Definition
\ref{def-parth2}}\label{fig-type_2_parthenogenisis}\end{center}\end{figure}

We refer to the subgraph $\Gamma_S =
\left<\{v_1,v_2\},(v_1,v_2)\right>$ as the {\em parthenogenic
bridge}, and say that a type 2 parthenogenic operation inserts
a parthenogenic bridge into a $2$-cracker. See Figure
\ref{fig-type_2_parthenogenisis} for an illustration.

\begin{figure}[h]\begin{center}\epsfig{scale=0.30, figure=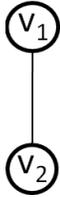}\\
\caption{A parthenogenic
bridge}\label{fig-parthenogenic_bridge}\end{center}\end{figure}
\begin{definition}
A {\em type 3 parthenogenic operation} is a function
$\mathcal{P}_3$ defined on the tuple $(\Gamma_1,a)$ where
$\Gamma_1 = \left<V_1,E_1\right>$ is a bridge graph and $a \in
V_1$ is a vertex incident to a $1$-cracker composing an edge $e_1
= (a,b) \in E_1$ and is adjacent to vertices $c$ and $d \in V_1$. This
function maps such a tuple onto another tuple $(\Gamma_D, a, v_1,
v_2)$ as follows
\begin{eqnarray}\mathcal{P}_3(\Gamma_1,a) & = & (\Gamma_D, a, v_1, v_2),\nonumber\end{eqnarray}

where $\Gamma_D = \left<V_D,E_D\right>$ and $\{a, v_1, v_2 \} \in
V_D$. The new set of vertices is $V_D = V_1 \;\cup\; \{v_1,v_2\}$.
The new set of edges is $E_D = (E_1 \backslash \{(a,c),(a,d)\})
\;\cup\;
\{(a,v_1),(a,v_2),(v_1,v_2),(v_1,c),(v_2,d)\}$.\label{def-parth3}\end{definition}

\begin{figure}[h]\begin{center}\hspace*{0cm}\epsfig{scale=0.35, figure=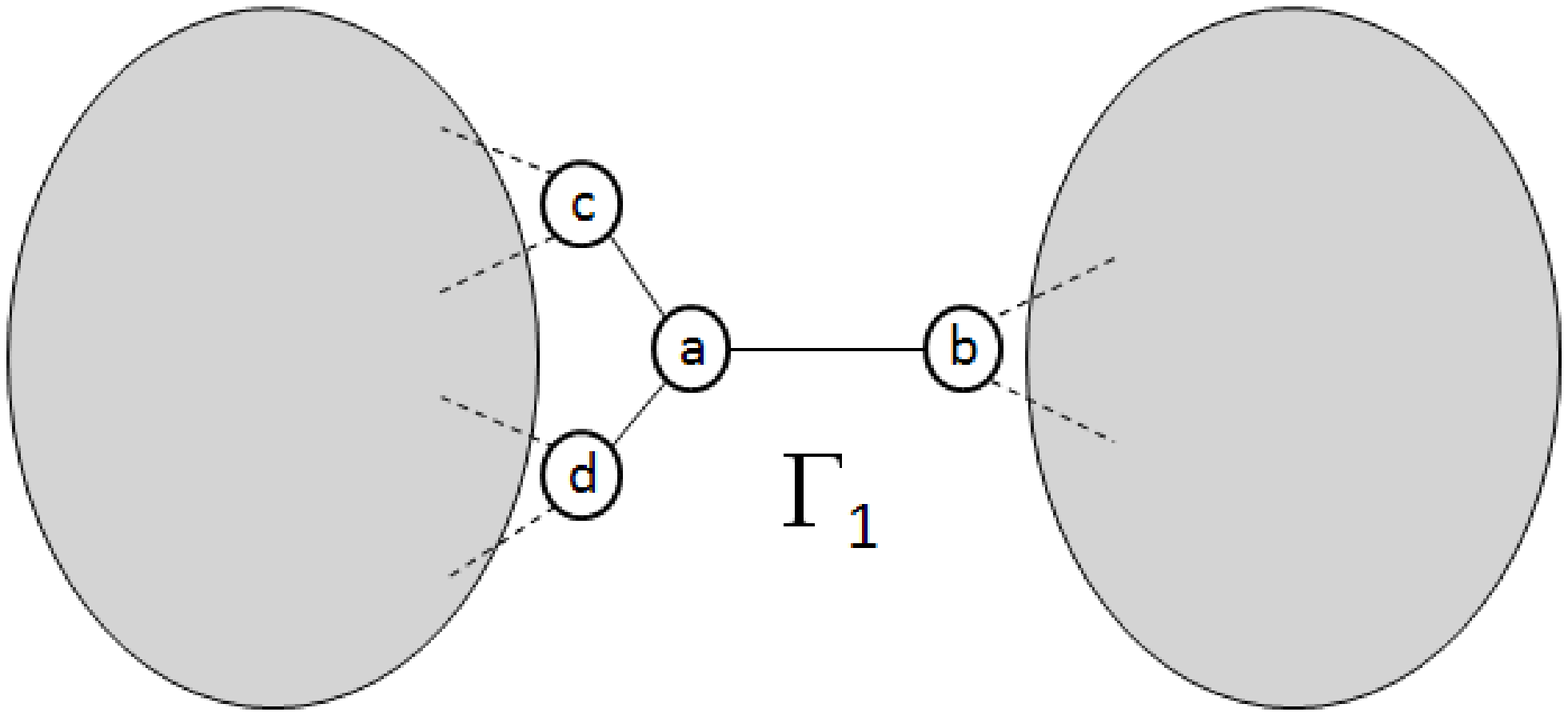}\hspace*{0.7cm}\epsfig{scale=0.35, figure=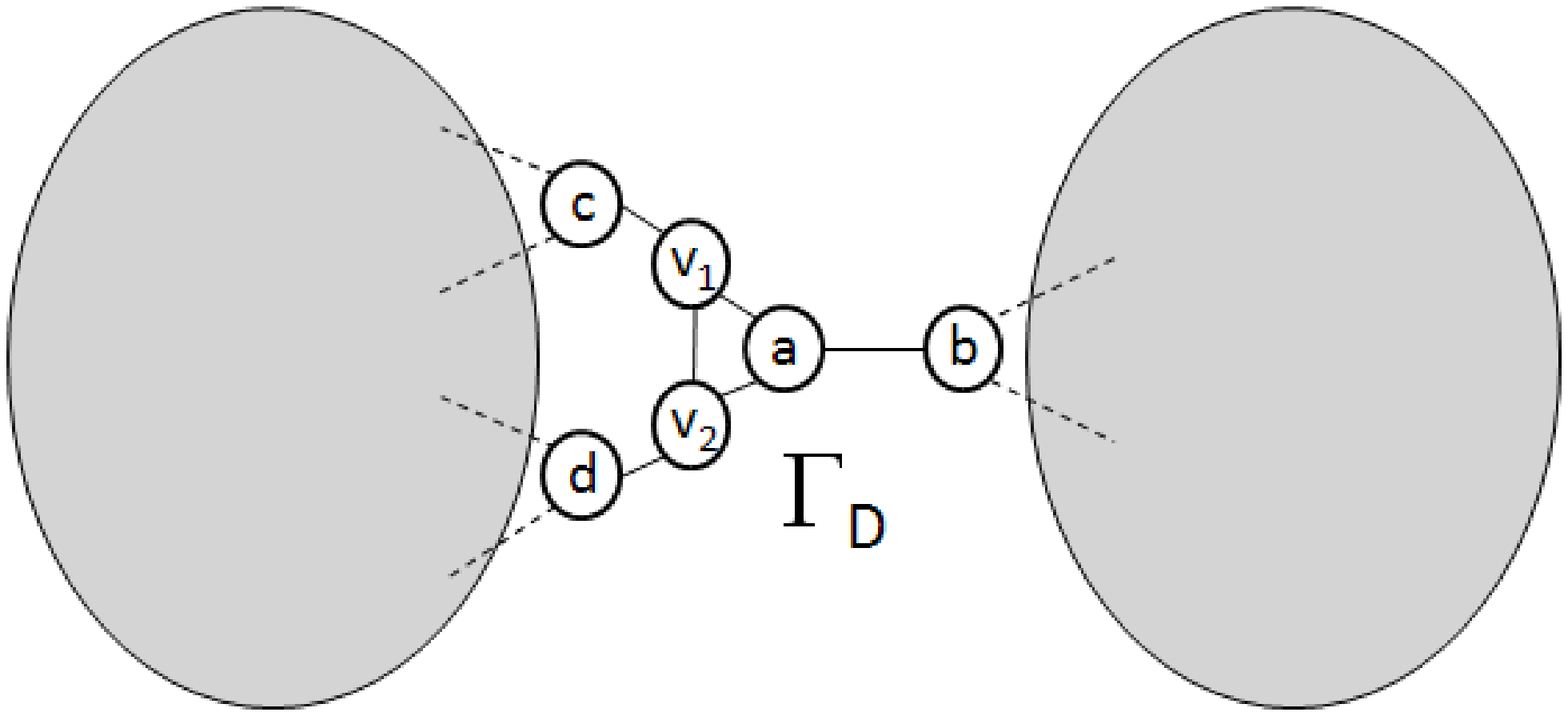}\\
\caption{Graphs $\Gamma_1$ and $\Gamma_D$ as described in
Definition
\ref{def-parth3}}\label{fig-type_3_parthenogenisis}\end{center}\end{figure}

We refer to the subgraph $\Gamma_S =
\left<\{a,v_1,v_2\},\{(a,v_1),(a,v_2),(v_1,v_2)\}\right>$ as the
{\em parthenogenic triangle}, and say that a type 3 parthenogenic
operation inserts a parthenogenic triangle next to the
1-cracker. See Figure \ref{fig-type_3_parthenogenisis} for an
illustration.

\begin{figure}[h]\begin{center}\epsfig{scale=0.30, figure=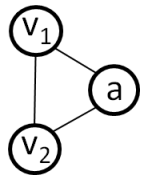}\\
\caption{A parthenogenic
triangle}\label{fig-parthenogenic_triangle}\end{center}\end{figure}

\begin{lemma}A child graph resulting from any of the six breeding operations is connected.\end{lemma}

\begin{proof}The nature of the six breeding operations is that the parent graphs
are mostly unaltered, and are changed only in a neighbourhood of the introduced cubic cracker. Therefore,
we can focus just on these areas. Since the parent graphs are (by definition) connected to
begin with, we only need to be concerned with which edges, present in the parent graphs,
are not present in the child graph.

For the cases of type 2 breeding, and types 1 and 2 parthenogenesis, only a single
edge from the parent graph (or from each of the parent graphs in the case of type 2 breeding)
is missing in the child graph. By definition this edge cannot be a 1-cracker, and therefore, the parent
graphs remain connected, and by construction it is clear that the adjoining cracker ensures
the child graph is also connected.

For type 1 breeding, only a single edge is removed from each parent graph. If neither edge is a 1-cracker,
then the argument in the previous paragraph can be used to show the child graph is connected. However, it is possible
that one or both removed edges could be 1-crackers. If so, the corresponding parent graphs become disconnected. However,
if this is the case, the graphs are reconnected by the introduction of vertices $v_1$ and $v_2$ (see Definition \ref{def-breed1}).
It is then clear by construction that the adjoining cracker ensures the child graph is also connected.

For type 3 breeding, a vertex is removed from both parent graphs. From Definition \ref{def-breed3},
we know that none of the edges adjacent to these two vertices constitute 1-crackers. Therefore, the
removal of these vertices cannot disconnect either graph. By construction it is then clear
that the adjoining cracker ensures the child graph is also connected.

Finally, for type 3 parthenogenesis, although two (adjacent) edges from the parent graph are missing
in the child graph, it is clear from the latter's construction that this can not result in a
disconnected descendant.\end{proof}

\subsection{Inverse breeding and inverse parthenogenic operations}

For some tuples $(\Gamma_D, e)$ where $\{e\} \in E_D$ is a
$1$-cracker, the inverse operation $\mathcal{B}^{-1}_1 (\Gamma_D ,
e) = ( \Gamma_1 , \Gamma_2, e_1, e_2)$ is well defined. In such a
case $ \{e \}$ is called an \emph{irreducible} $1$-cracker. If not
$ \{e \}$ will be called a \emph{reducible} $1$-cracker. Similarly
if the inverse operation $\mathcal{B}^{-1}_2 (\Gamma_D , e_3 ,
e_4) = ( \Gamma_1 , \Gamma_2, e_1, e_2)$ is well defined, where $
\{e_3 , e_4\} \in E_D$ is a $2$-cracker, the 2-cracker is called
an irreducible $2$-cracker and reducible $2$-cracker otherwise. We
will show later that the inverse operation $\mathcal{B}^{-1}_3
(\Gamma_D , e_1, e_2, e_3) = ( \Gamma_1 , \Gamma_2, v_1, v_2, a,
b, c, d, e, f)$ is always defined where $\{ e_1, e_2, e_3 \} \in E_D$
is a $3$-cracker. Therefore every 3-cracker is irreducible.
\begin{definition}
Whenever a cubic cracker is irreducible one of the equations
\eqref{eq_invb1}, \eqref{eq_invb2} and \eqref{eq_invb3} defines
the corresponding inverse breeding operation $\mathcal{B}^{-1}_1
(.)$ , $\mathcal{B}^{-1}_2 (.)$ or $\mathcal{B}^{-1}_3 (.)$. The
two cubic graphs $(\Gamma_1, \Gamma_2)$ from the tuple produced by
these operations are parents of $\Gamma_D$. In particular,
\begin{equation} \label{eq_invb1}
\mathcal{B}^{-1}_1 (\Gamma_D , e) = ( \Gamma_1 , \Gamma_2, e_1,
e_2)
\end{equation}
where $\Gamma_D$, $e$, $\Gamma_1$, $\Gamma_2$, $e_1$ and $e_2$
are defined in Definition \ref{def-breed1}. Similarly,
\begin{equation} \label{eq_invb2}
\mathcal{B}^{-1}_2 (\Gamma_D , e_3 , e_4) = ( \Gamma_1 , \Gamma_2,
e_1, e_2)
\end{equation}
where $\Gamma_D$, $e_3$, $e_4$, $\Gamma_1$, $\Gamma_2$, $e_1$
and $e_2$ are defined in Definition \ref{def-breed2}. Also,
\begin{equation} \label{eq_invb3}
\mathcal{B}^{-1}_3 (\Gamma_D , e_1, e_2, e_3) = ( \Gamma_1 ,
\Gamma_2, v_1, v_2, a, b, c, d, e, f)
\end{equation}
where $\Gamma_D$, $e_1$, $e_2$, $e_3$, $\Gamma_1$, $\Gamma_2$,
$v_1$, $v_2$, $a$, $b$, $c$, $d$, $e$ and $f$ are defined in
Definition \ref{def-breed3}.
\end{definition}
Similarly, inverse parthenogenic operations can be defined as
follows.
\begin{definition}
Equations \eqref{eq_invpath1}, \eqref{eq_invpath2} and
\eqref{eq_invpath3} define the corresponding inverse
parthenogenic operations $\mathcal{P}^{-1}_1 (.)$ ,
$\mathcal{P}^{-1}_2 (.)$ or $\mathcal{P}^{-1}_3 (.)$. The cubic
graph $(\Gamma_1)$ from the tuple produced by these operations is a
parent of $\Gamma_D$. In particular,
\begin{equation} \label{eq_invpath1}
\mathcal{P}^{-1}_1 (\Gamma_D , v_1 , v_4) = ( \Gamma_1 , e_1),
\end{equation}
where $\Gamma_D$, $v_1$, $v_4$, $\Gamma_1$ and $e_1$ are defined
in Definition \ref{def-parth1}. Similarly,

\begin{equation} \label{eq_invpath2}
\mathcal{P}^{-1}_2 (\Gamma_D , v_1 , v_2) = ( \Gamma_1, e_1, e_2),
\end{equation}
where $\Gamma_D$, $v_1$, $v_2$, $\Gamma_1$, $e_1$ and $e_2$ are
defined in Definition \ref{def-parth2}. Also,

\begin{equation} \label{eq_invpath3}
\mathcal{P}^{-1}_3 (\Gamma_D , a , v_1, v_2) = (\Gamma_1 , a),
\end{equation} Where $\Gamma_D$, $a$, $\Gamma_1$,
$\Gamma_2$, $v_1$ and $v_2$ are defined in Definition
\ref{def-parth3}.

Collectively, we refer to the three inverse breeding operations and
the three inverse parthenogenic operations as the six {\em inverse operations}.\end{definition}

It is important to note that the six breeding operations and the six inverse
operations presented here are not entirely new, and have been used in various forms
in other cubic graph generation routines. For example, type 2 parthenogenesis induces
an {\em H-subgraph} which is well-studied in literature (e.g. see Ore \cite{ore}).
Type 1 inverse parthenogenesis and type 1 inverse breeding appear as
Operation $O_2(K_4)$ and Operations $\mathcal{R}$ respectively in Ding and
Kanno \cite{dingkanno}. Types 1 and 2 parthenogenesis appear
in Brinkmann \cite{brinkmann}. All of the six breeding operations except type 3
breeding appear in some sense as generating rules in Batagelj \cite{batagelj},
specifically generating rules P1, P2, P3, P4 an P8. However, all of the above works
involve growing the complexity of a single graph by the evolution of subgraphs,
rather than combining several cubic graphs together. In addition, the operations in the
above works that are analogous to our parthenogenic operations are not confined to
the same conditions as ours (that is, they must occur on $1$-crackers and
$2$-crackers). Little consideration is given in the above works to procedures that are analogous to
our inverse operations. The benefits and potency of the particular set of breeding and
inverse operations that we have detailed in this section are demonstrated in
the following section.

\section{Results}\label{sec-results}

The following three propositions relate to the different possible
methods of creation of cubic crackers by the six breeding operations,
and are used in the proof of the main theorem for this section, Theorem
\ref{theorem-anygraph}.

\begin{proposition}
Any descendant involving a 1-cracker can be obtained from either
type 1 breeding, type 1 parthenogenesis, or type 3
parthenogenesis. \label{prop-1cracker}\end{proposition}
\begin{proof} Consider a descendant $\Gamma_D =
\left<V_D,E_D\right>$ containing a 1-cracker comprising an edge
$(v_1,v_2)$. Since $\Gamma_D$ is cubic, $v_1$ and $v_2$ will both
be adjacent to two more vertices, say $\{a,b\}$ and $\{c,d\}$
respectively. Since the 1-cracker comprises edge $(v_1,v_2)$, we
know that $\{a,b\}$ and $\{c,d\}$ are disjoint sets. Then, we can
consider two cases.

{\bf Case 1:} The edges $(a,b)$ and $(c,d)$ are not present in
$\Gamma_D$. In this case, the 1-cracker is irreducible.
Suppose the bridge $(v_1,v_2)$ is removed from
$\Gamma_D$, separating the graph into two subgraphs, $\Gamma_{S1}
= \left<V_{S1},E_{S1}\right>$ and $\Gamma_{S2} =
\left<V_{S2},E_{S2}\right>$. Without loss of generality, we assume
that $v_1 \in V_{S1}$, and $v_2 \in V_{S2}$. Then, we define a
cubic graph $\Gamma_1 = \left<V_1,E_1\right>$, where $V_1 = V_{S1}
\backslash v_1$ and $E_1 = (E_{S1} \backslash \{(a,v_1),(b,v_1)\})
\;\cup\; (a,b)$. Similarly, we define a second cubic graph
$\Gamma_2 = \left<V_2,E_2\right>$, where $V_2 = V_{S2} \backslash
v_2$ and $E_2 = (E_{S2} \backslash \{(c,v_2),(d,v_2)\}) \;\cup\;
(c,d)$. Then, $\Gamma_D$ can be obtained from the type 1 breeding
operation $\mathcal{B}_1(\Gamma_1,\Gamma_2,(a,b),(c,d))$.

\begin{figure}[h]\begin{center}\hspace*{0cm}\epsfig{scale=0.35, figure=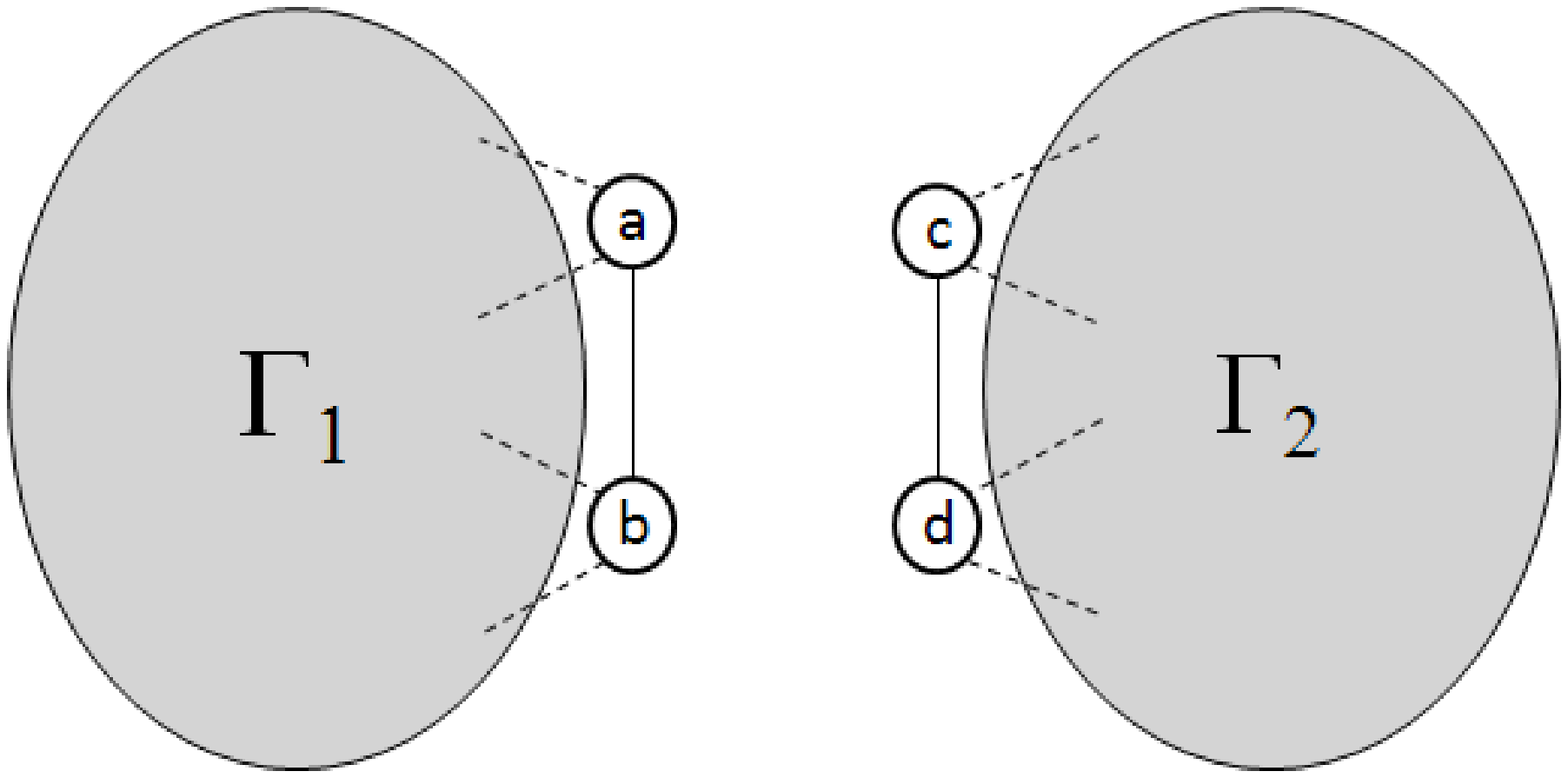}\hspace*{1.0cm}\epsfig{scale=0.35, figure=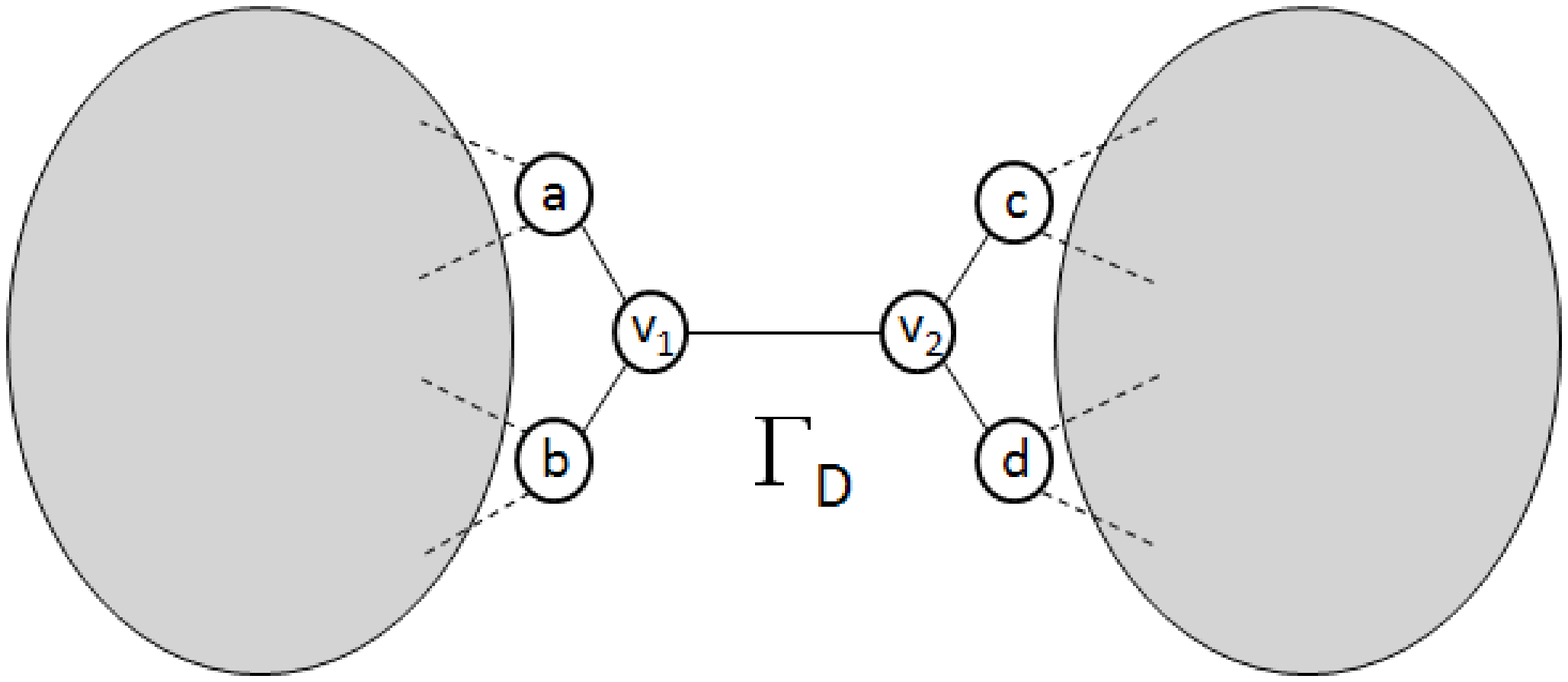}\\
\caption{Graphs $\Gamma_1$, $\Gamma_2$ and $\Gamma_D$ as described
in Case 1 of Proposition
\ref{prop-1cracker}}\label{fig-1_cracker_case_1}\end{center}\end{figure}

{\bf Case 2:} At least one of the edges $(a,b)$ or $(c,d)$ is present in
$\Gamma_D$. In this case, the 1-cracker cannot be
obtained from a type 1 breeding operation, as such an operation
would remove these edges. A 1-cracker of this type is, therefore,
reducible. If both edges are present, we can focus on either one.
Without loss of generality, we will assume that
edge $(a,b) \in E_D$. Then, since $\Gamma_D$ is cubic, and
vertices $a$ and $b$ are both adjacent to vertex $v_1$ and to each
other, they will also be adjacent to one more vertex each, say
vertices $e$ and $f$ respectively. Note that is is possible that
$e = f$, so we need to consider the cases separately.

{\bf Case 2.1:} $e = f$. Then, both edges $(a,e)$ and $(b,e)$ are
in $E_D$, as seen in the right panel of Figure \ref{fig-1_cracker_case_21}. Since
$\Gamma_D$ is cubic, vertex $e$ is adjacent to a third vertex, say
$g$. Clearly, edge $(e,g)$ is a bridge, since edge $(v_1,v_2)$ is
a bridge. Then, we define a cubic graph $\Gamma_1 =
\left<V_1,E_1\right>$, where $E_1 = (E_D \backslash
\{(v_1,v_2),(a,v_1),(b,v_1),(a,b),(a,e),(b,e),(e,g)\}) \;\cup\;
(g,v_2)$ and $V_1 = V_D \backslash \{v_1,a,b,e\}$. We can then
obtain $\Gamma_D$ from the type 1 parthenogenic operation
$\mathcal{P}_1(\Gamma_1,(g,v_2))$.

\begin{figure}[h]\begin{center}\hspace*{0cm}\epsfig{scale=0.35, figure=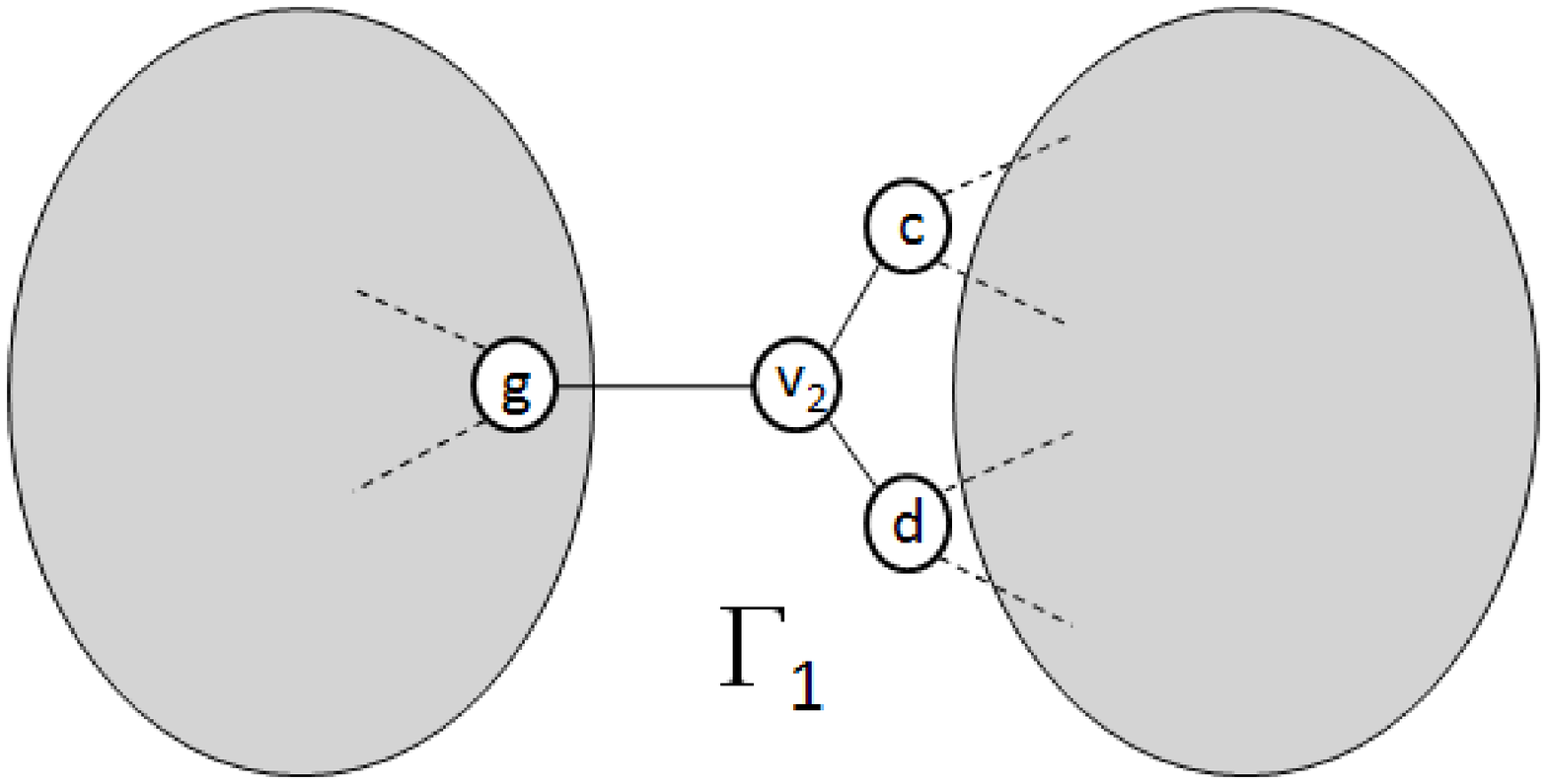}\hspace*{1.0cm}\epsfig{scale=0.35, figure=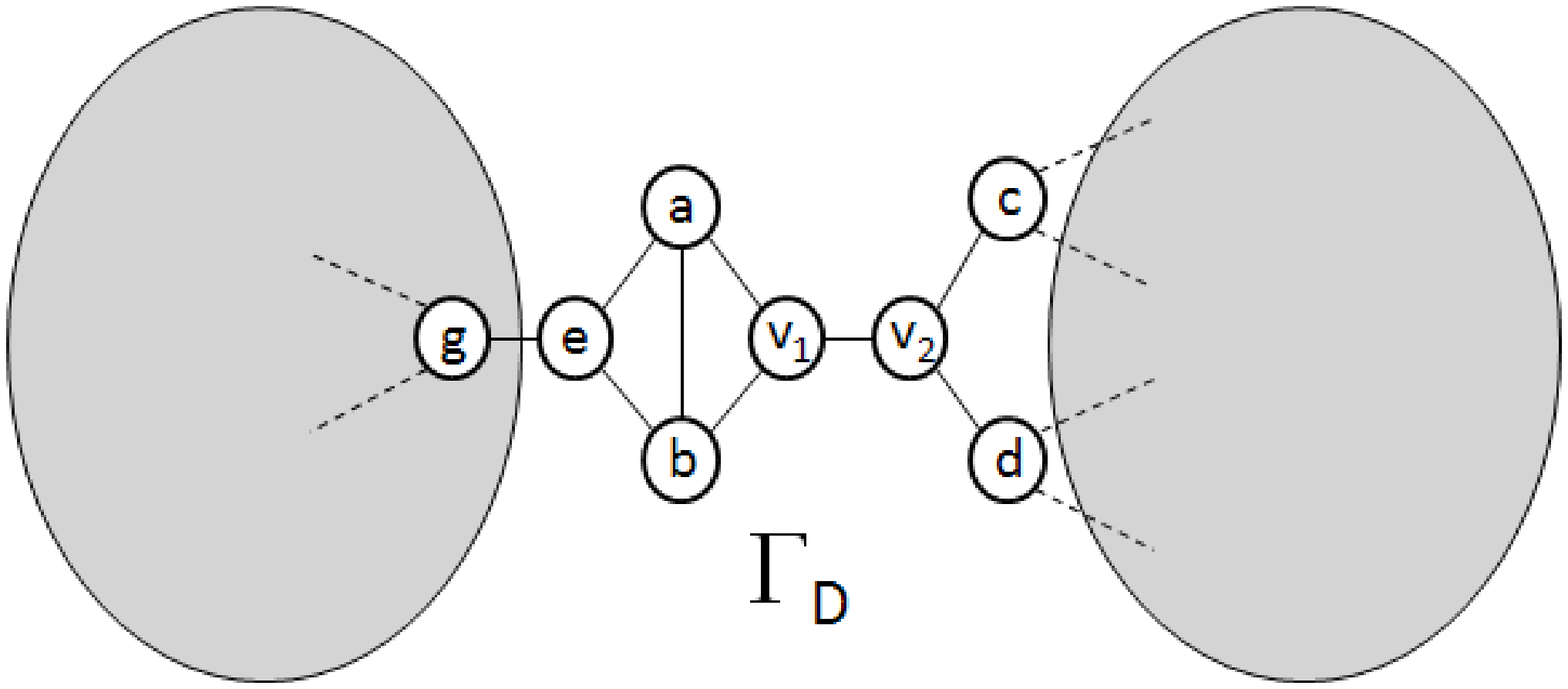}\\
\caption{Graphs $\Gamma_1$ and $\Gamma_D$ as described in Case 2.1
of Proposition
\ref{prop-1cracker}}\label{fig-1_cracker_case_21}\end{center}\end{figure}

{\bf Case 2.2:} $e \neq f$. Then, as illustrated in Figure
\ref{fig-1_cracker_case_22}, we define a cubic graph $\Gamma_1 =
\left<V_1,E_1\right>$, where $E_1 = (E_D \backslash \{(a,v_1),
(b,v_1),(a,b),(a,e),(b,f)\}) \;\cup\; \{(e,v_1),(f,v_1)\}$ and
$V_1 = V_D \backslash \{a,b\}$. Then, we can obtain $\Gamma_D$
from the type 3 parthenogenic operation
$\mathcal{P}_3(\Gamma_1,v_1)$.
\end{proof}
\begin{figure}[h]\begin{center}\hspace*{0cm}\epsfig{scale=0.35, figure=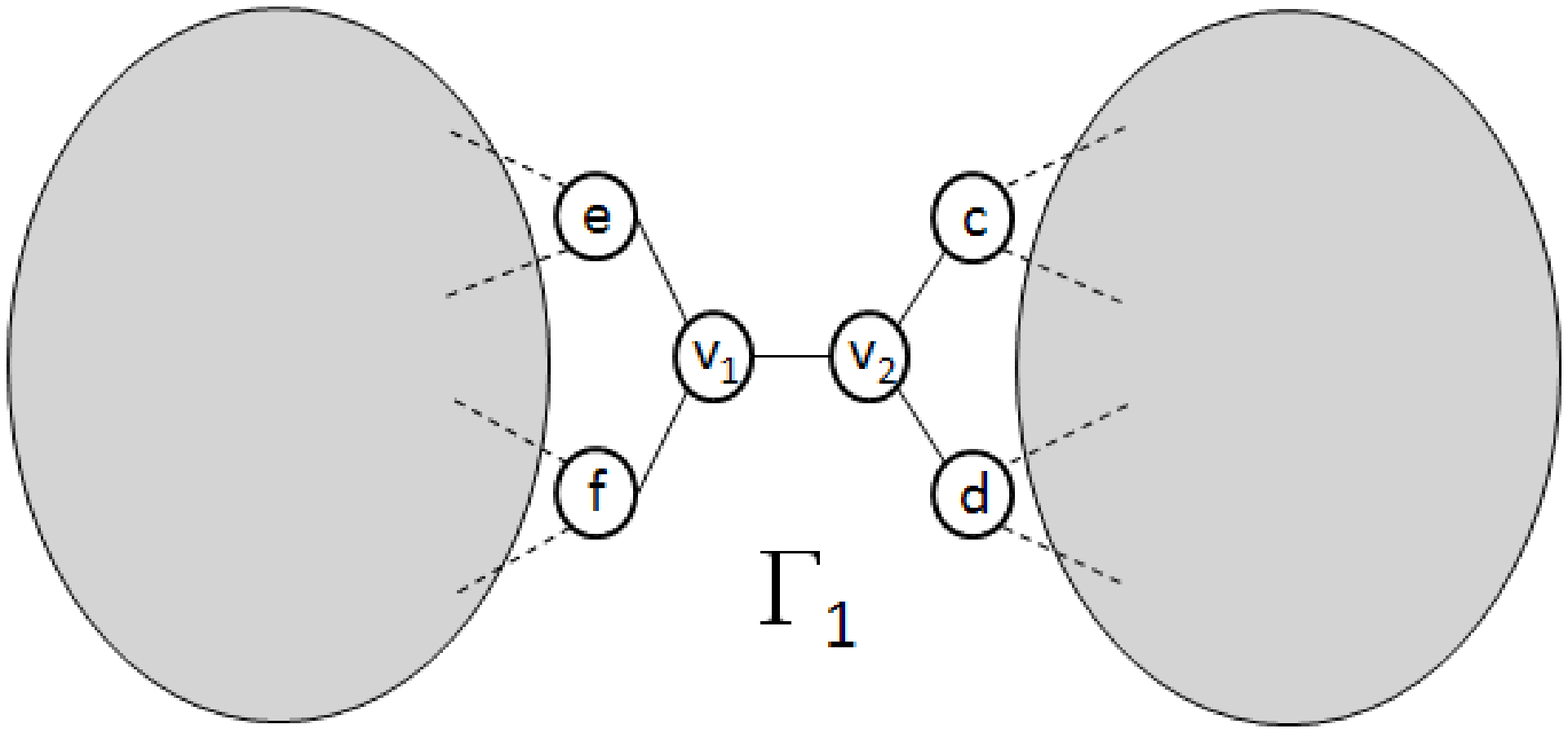}\hspace*{0.8cm}\epsfig{scale=0.35, figure=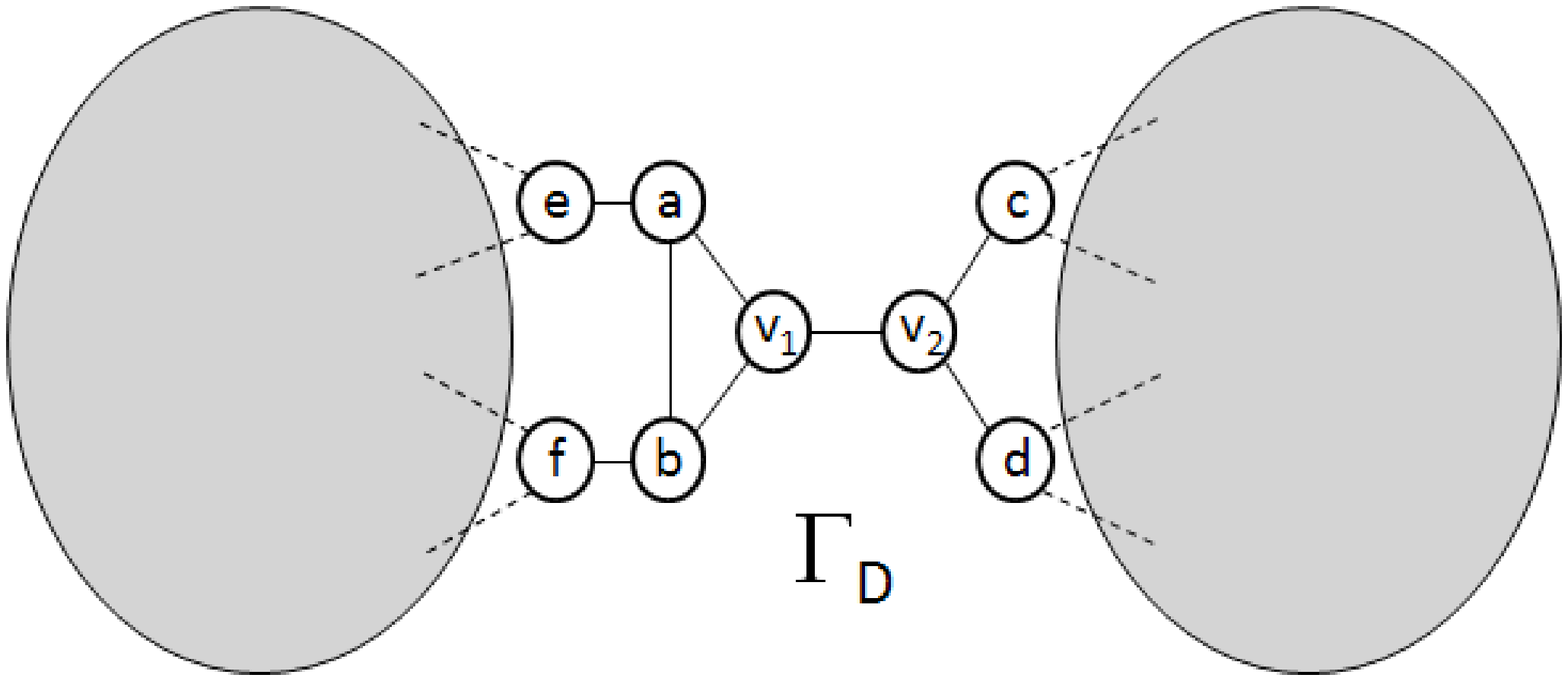}\\
\caption{Graphs $\Gamma_1$ and $\Gamma_D$ as described in Case 2.2
of Proposition
\ref{prop-1cracker}}\label{fig-1_cracker_case_22}\end{center}\end{figure}

\begin{proposition}Any descendant involving a $2$-cracker can be obtained from either type 1 breeding, type 2 breeding, type 1 parthenogenesis, type 2 parthenogenesis, or type 3 parthenogenesis.\label{prop-2cracker}\end{proposition}

\begin{proof} Consider a descendant $\Gamma_D = \left<V_D,E_D\right>$
containing a $2$-cracker comprising edges $(v_1,v_2)$ and
$(v_3,v_4)$. We consider two cases.

{\bf Case 1:} Neither edge $(v_1,v_3)$ nor $(v_2,v_4)$ are not
present in $\Gamma_D$, as illustrated in Figure
\ref{fig-2_cracker_case_1}. In this case, the 2-cracker is
irreducible. Suppose the edges $(v_1,v_2)$
and $(v_3,v_4)$ are removed from $\Gamma_D$, separating the graph
into two subgraphs $\Gamma_{S1} = \left<V_{S1},E_{S1}\right>$ and
$\Gamma_{S2} = \left<V_{S2},E_{S2}\right>$. Without loss of
generality, we assume that $v_1, v_3 \in V_{S1}$ and $v_2, v_4 \in V_{S2}$.
Then, we define a cubic
graph $\Gamma_1 = \left<V_1,E_1\right>$, where $V_1 = V_{S1}$ and
$E_1 = E_{S1} \;\cup\; (v_1,v_3)$. Similarly, we define a second
cubic graph $\Gamma_2 = \left<V_2,E_2\right>$, where $V_2 =
V_{S2}$ and $E_2 = E_{S2} \;\cup\; (v_2,v_4)$. Then, $\Gamma_D$
is obtained from the type 2 breeding operation
$\mathcal{B}_2(\Gamma_1,\Gamma_2,(v_1,v_3),(v_2,v_4))$.

\begin{figure}[h]\begin{center}\hspace*{0cm}\epsfig{scale=0.35, figure=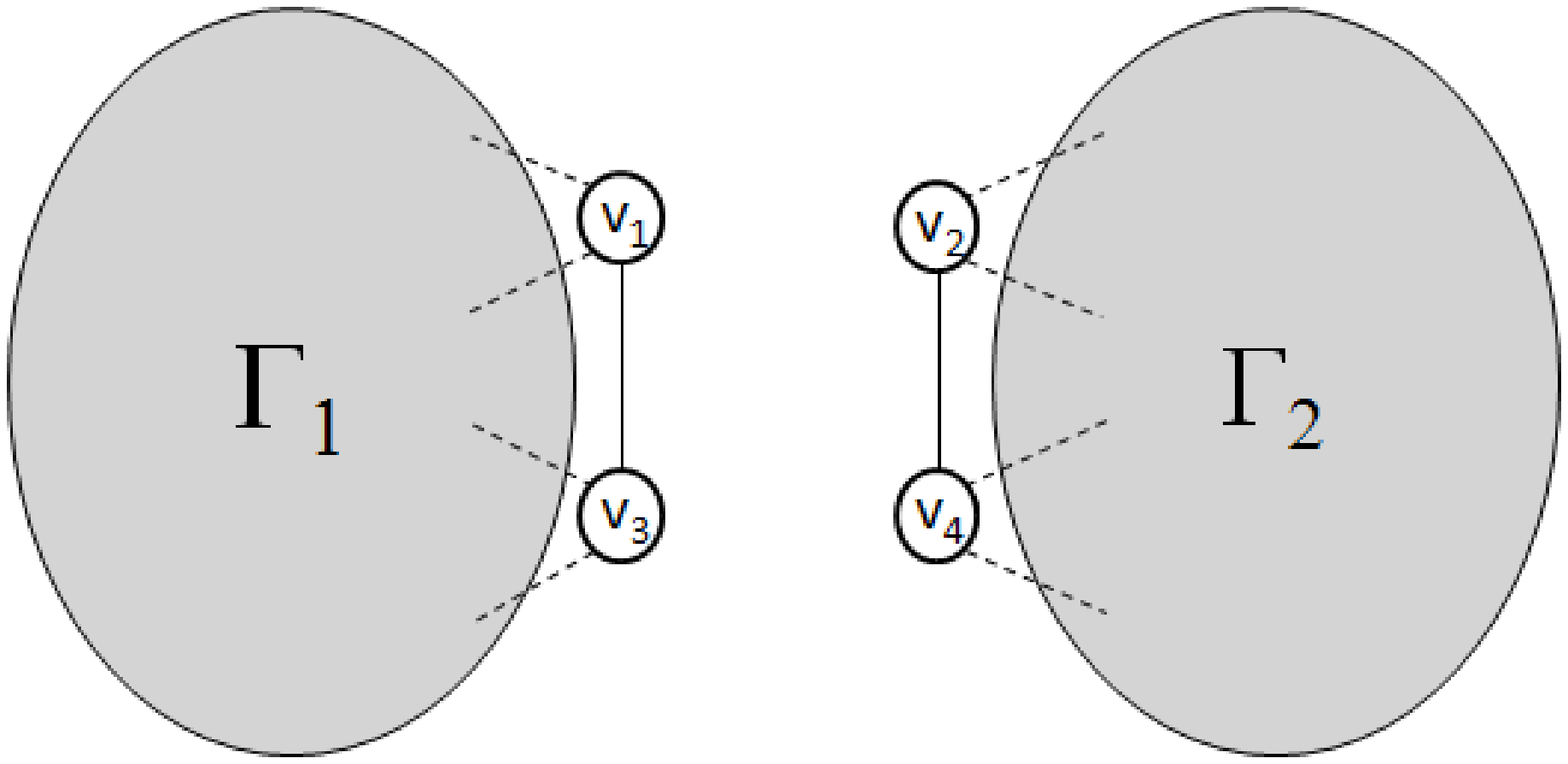}\hspace*{1.0cm}\epsfig{scale=0.35, figure=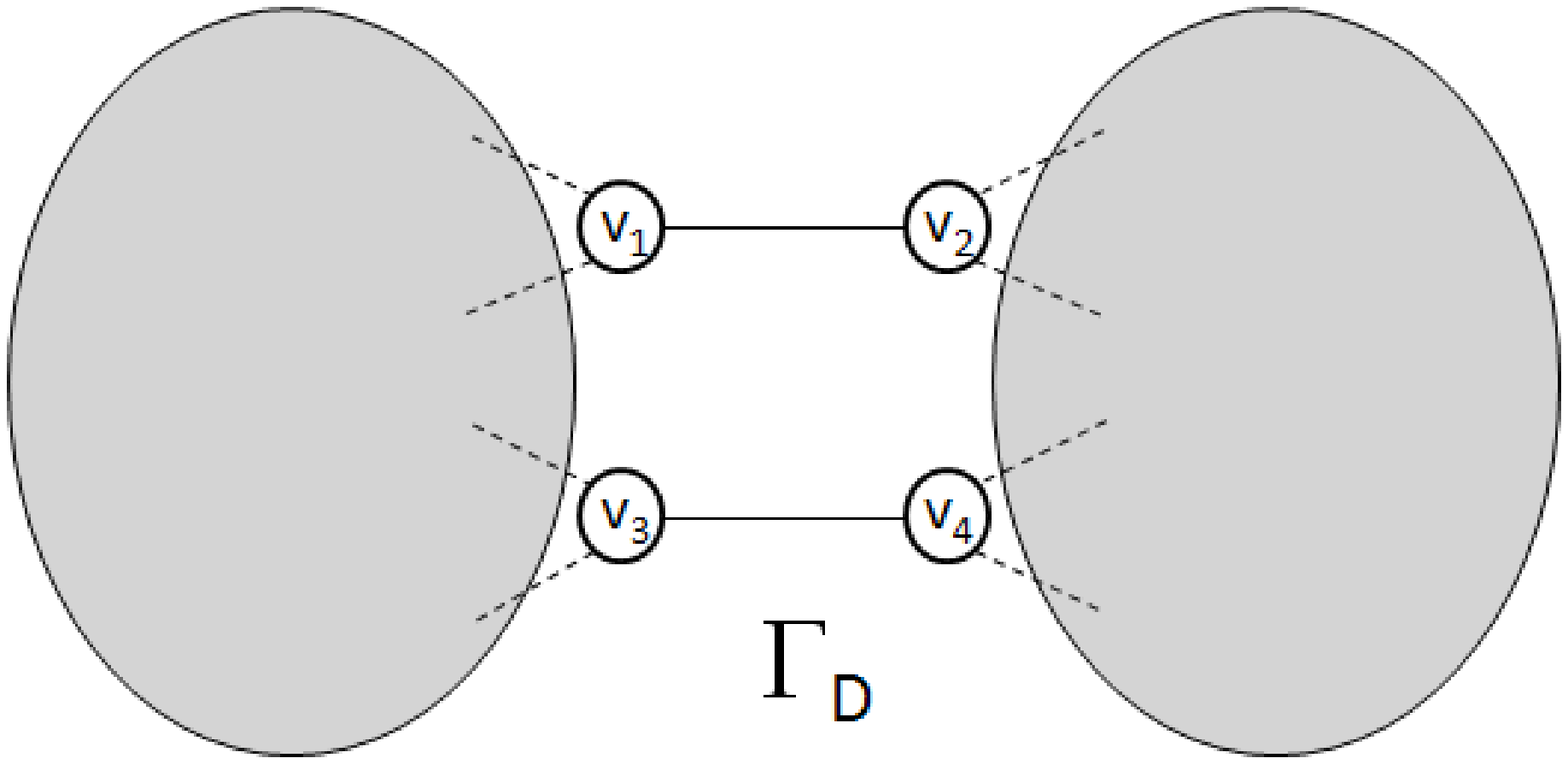}\\
\caption{Graphs $\Gamma_1$, $\Gamma_2$ and $\Gamma_D$ as described
in Case 1 of Proposition
\ref{prop-2cracker}}\label{fig-2_cracker_case_1}\end{center}\end{figure}

{\bf Case 2:} At least one of the edges $(v_1,v_3)$ or $(v_2,v_4)$ is
present in $\Gamma_D$. In this case, the 2-cracker cannot be
obtained from a type 2 breeding operation, as such an operation
would remove these edges. A 2-cracker of this type is therefore
reducible. If both edges are present, we can focus on either
one. Without loss of generality, we will assume that edge
$(v_1,v_3) \in E_D$. Then, since $\Gamma_D$ is cubic, and vertex
$v_1$ is adjacent to vertices $v_2$ and $v_3$, it will be adjacent
to one more vertex, say vertex $a$. Similarly, since vertex $v_3$
is adjacent to vertices $v_1$ and $v_4$, it will be adjacent to
one more vertex, say vertex $b$. Note that it is possible that $a
= b$, so we need to consider the cases separately.

{\bf Case 2.1:} $a = b$. In this case, edge $(v_1,a) \in E_D$ and edge
$(v_3,a) \in E_D$, as illustrated in Figure
\ref{fig-2_cracker_case_21}. Since $\Gamma_D$ is cubic, vertex $a$
is adjacent to a third vertex, say $c$. Clearly, edge $(a,c)$ is a
bridge. Then, we define a cubic graph $\Gamma_1 =
\left<V_1,E_1\right>$, where $V_1 = V_D \backslash \{v_1,v_3\}$
and $E_1 = E_D \backslash
\{(v_1,v_2),(v_3,v_4),(v_1,v_3),(a,v_1),(a,v_3)\} \;\cup\;
\{(a,v_2),(a,v_4)\}$. Then, we can obtain $\Gamma_D$ from the type
3 parthenogenic operation $\mathcal{P}_3(\Gamma_1,a)$. Note that
this case is essentially the same as Case 2.2 in Proposition
\ref{prop-1cracker}.

\begin{figure}[h]\begin{center}\hspace*{0cm}\epsfig{scale=0.35, figure=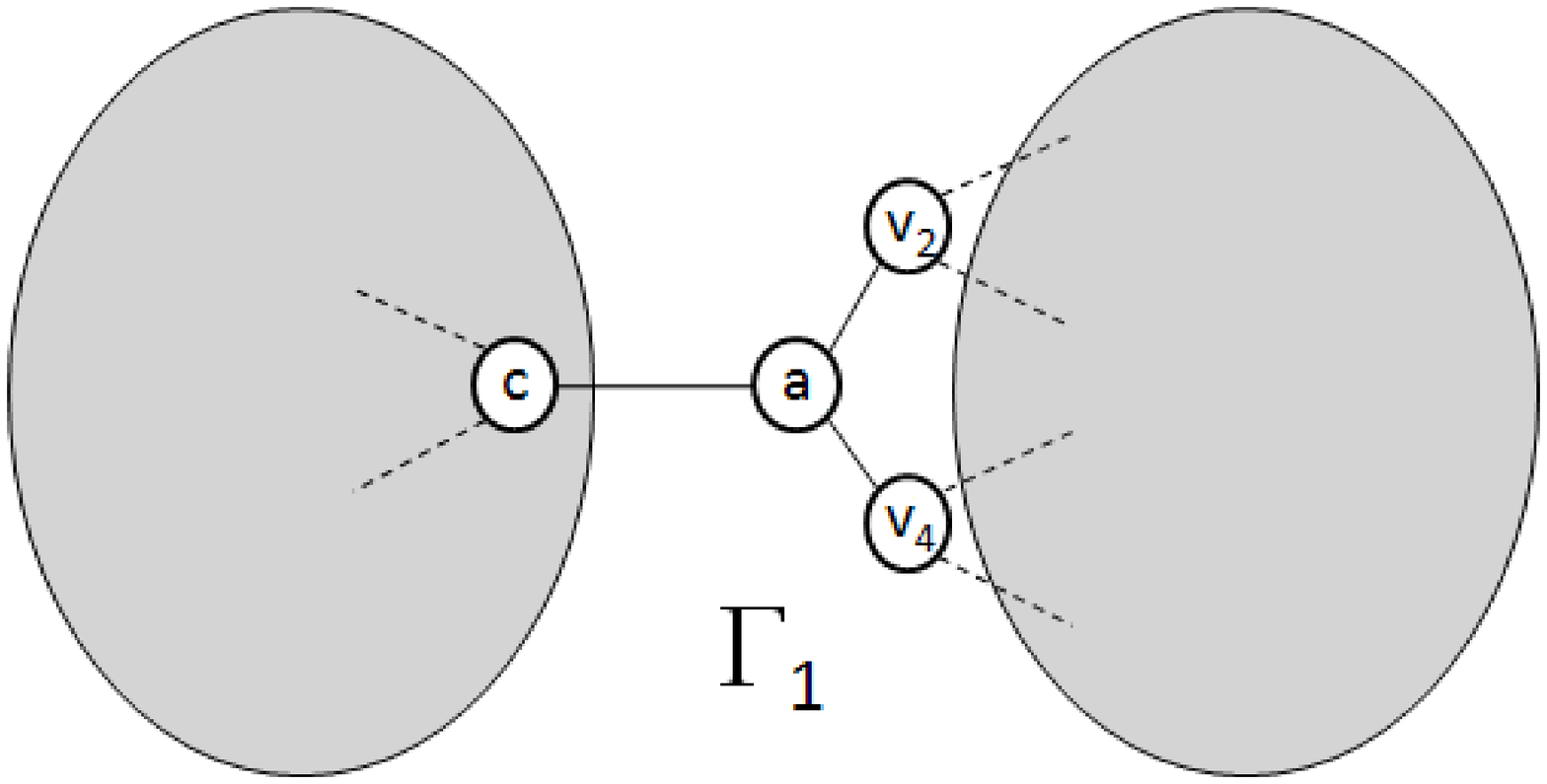}\hspace*{1.0cm}\epsfig{scale=0.35, figure=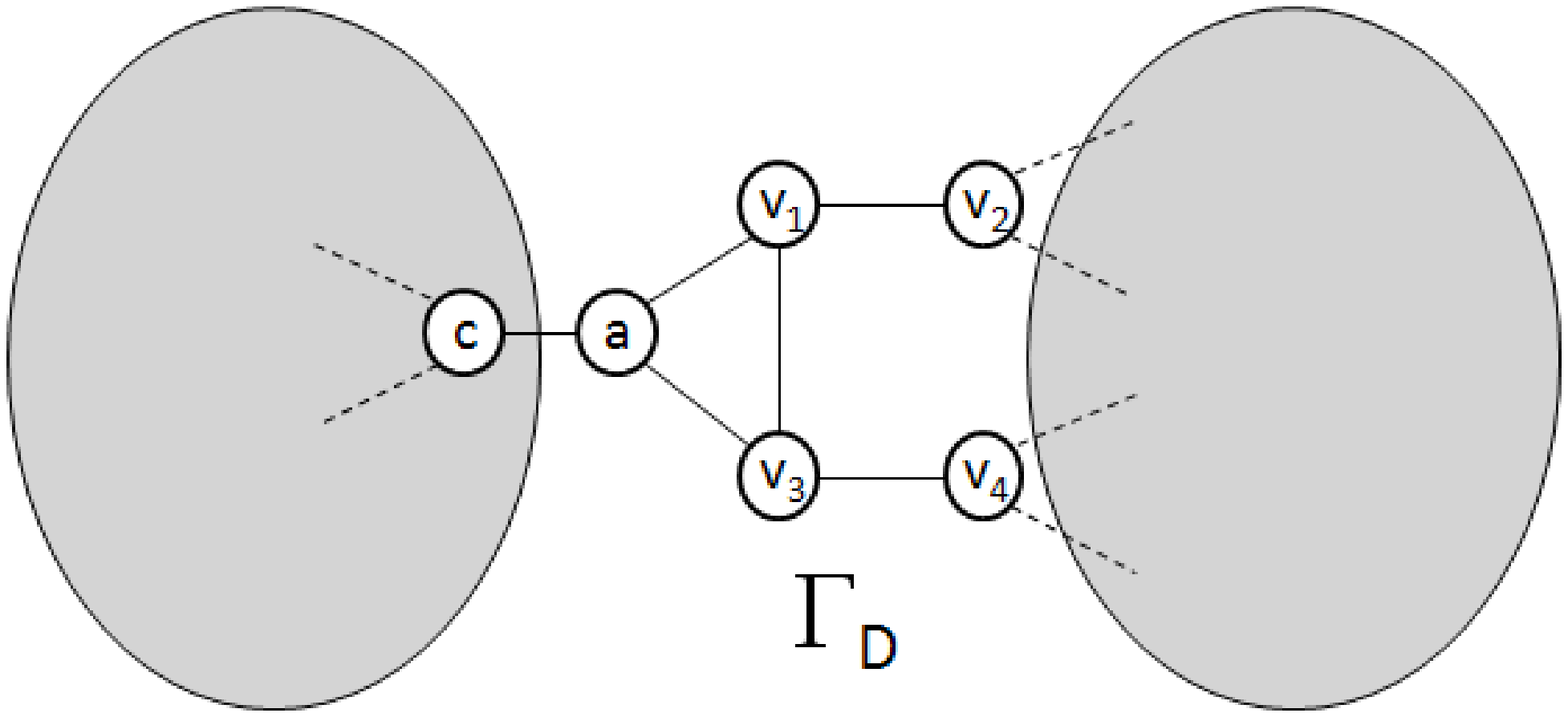}\\
\caption{Graphs $\Gamma_1$ and $\Gamma_D$ as described in Case 2.1
of Proposition
\ref{prop-2cracker}}\label{fig-2_cracker_case_21}\end{center}\end{figure}

{\bf Case 2.2:} $a \neq b$. It is obvious that the non-adjacent edges $(v_1,a)$ and $(v_3,b)$ form a cutset. This implies that either both edges are $1$-crackers, or together they form a $2$-cracker. The former situation is covered in Proposition \ref{prop:graphs_With_1_cracker} and $\Gamma_D$ can be obtained from either
type 1 breeding, type 1 parthenogenesis, or type 3
parthenogenesis. For the latter situation, let us define a cubic graph
$\Gamma_1 = \left<V_1,E_1\right>$, where $V_1 = V_D \backslash
\{v_1,v_3\}$ and $E_1 = E_D \backslash
\{(v_1,v_2),(v_3,v_4),(v_1,v_3),(a,v_1),(b,v_3)\} \;\cup\;
\{(a,v_2),(b,v_4)\}$, as illustrated in Figure
\ref{fig-2_cracker_case_22}. We can then obtain $\Gamma_D$ from the type
2 parthenogenic operation
$\mathcal{P}_2(\Gamma_1,(a,v_2),(b,v_4))$. \end{proof}
\begin{figure}[h]\begin{center}\hspace*{0cm}\epsfig{scale=0.35, figure=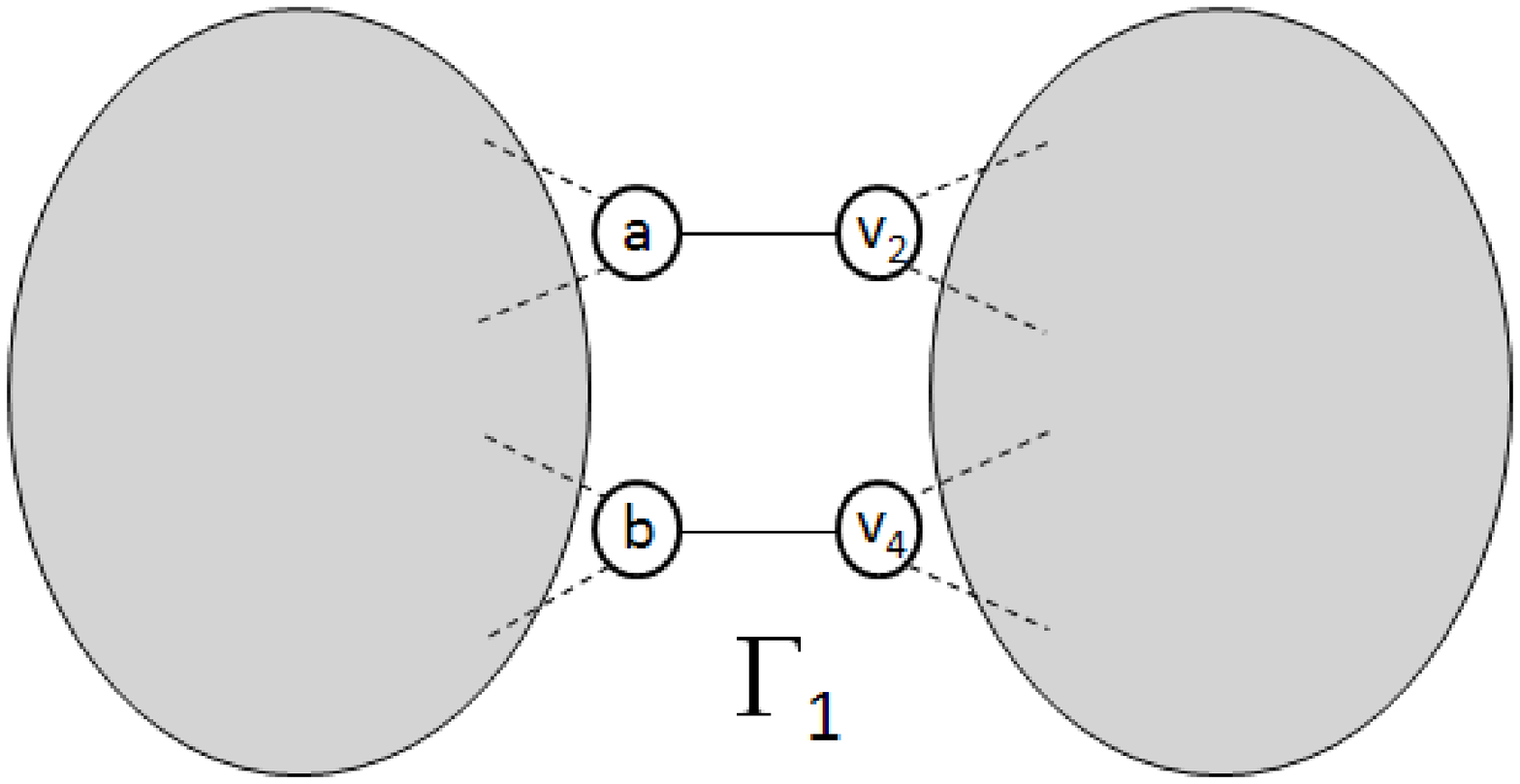}\hspace*{1.0cm}\epsfig{scale=0.35, figure=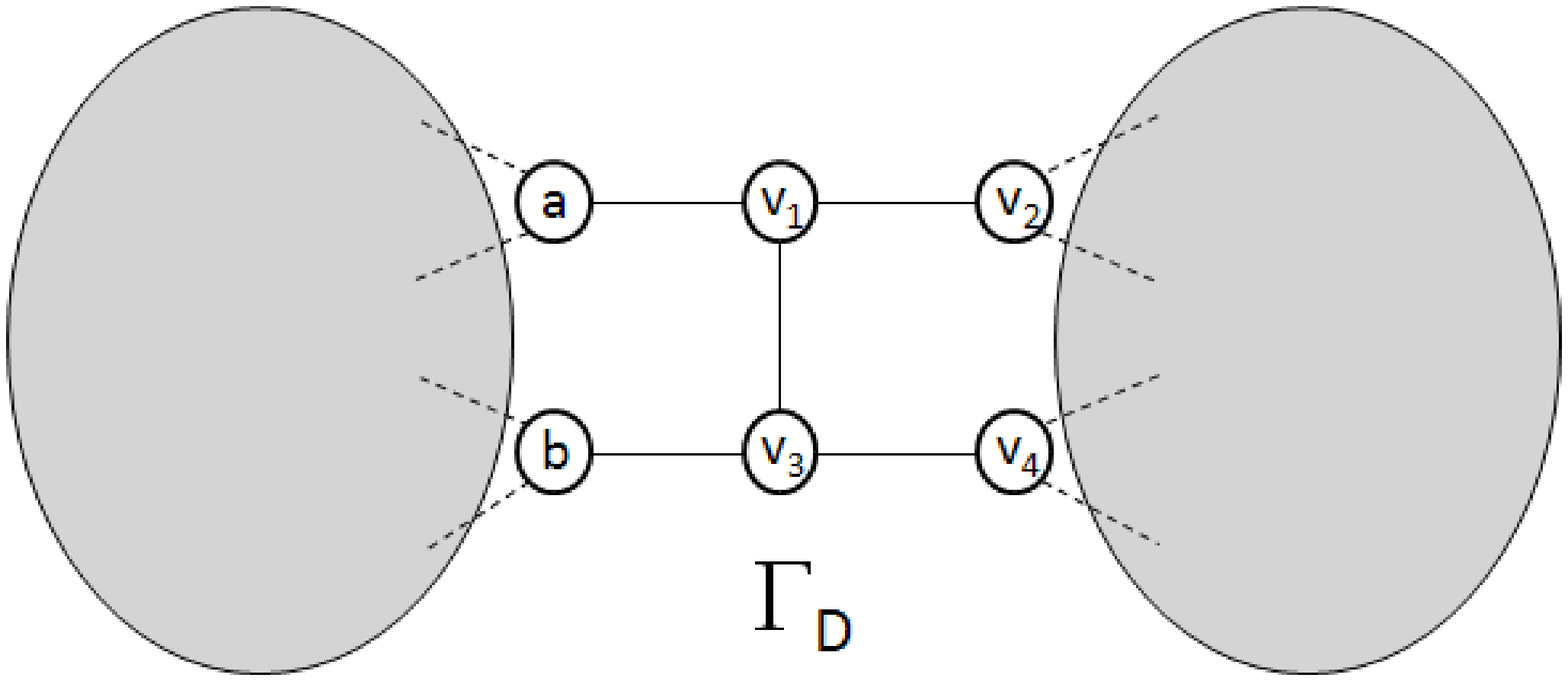}\\
\caption{Graphs $\Gamma_1$ and $\Gamma_D$ as described in Case 2.2
of Proposition
\ref{prop-2cracker}}\label{fig-2_cracker_case_22}\end{center}\end{figure}

\begin{proposition}Any descendant involving a $3$-cracker can be obtained from type 3
breeding.\label{prop-3cracker}\end{proposition}

\begin{proof} Consider a descendant $\Gamma_D = \left<V_D,E_D\right>$
containing a $3$-cracker comprising edges $(v_1,v_2)$, $(v_3,v_4)$
and $(v_5,v_6)$, as illustrated in Figure
\ref{fig-3_cracker_case_1}. Then, suppose the edges $(v_1,v_2)$,
$(v_3,v_4)$ and $(v_5,v_6)$ are removed from $\Gamma_D$,
separating the graph into two subgraphs $\Gamma_{S1} =
\left<V_{S1},E_{S1}\right>$ and $\Gamma_{S2} =
\left<V_{S2},E_{S2}\right>$. Without loss of generality, we assume
that $\{v_1, v_3, v_5 \}\in V_{S1}$ and $\{v_2, v_4, v_6 \} \in
V_{S1}$. Then, we introduce a new vertex $v_7$, and define a cubic
graph $\Gamma_1 = \left<V_1,E_1\right>$, where $V_1 = V_{S1}
\;\cup\; \{v_7\}$ and $E_1 = E_{S1} \;\cup\;
\{(v_1,v_7),(v_3,v_7),(v_5,v_7)\}$. Similarly, we introduce a new
vertex $v_8$, and define a second cubic graph $\Gamma_2 =
\left<V_2,E_2\right>$, where $V_2 = V_{S2} \;\cup\; \{v_8\}$ and
$E_2 = E_{S2} \;\cup\; \{(v_2,v_8),(v_4,v_8),(v_6,v_8)\}$. Then,
we can obtain $\Gamma_D$ from the type 3 breeding operation
$\mathcal{B}_3(\Gamma_1,\Gamma_2,v_7,v_8,v_1,v_3,v_5,v_2,v_4,v_6)$.
\end{proof}
\begin{figure}[h]\begin{center}\hspace*{0cm}\epsfig{scale=0.35, figure=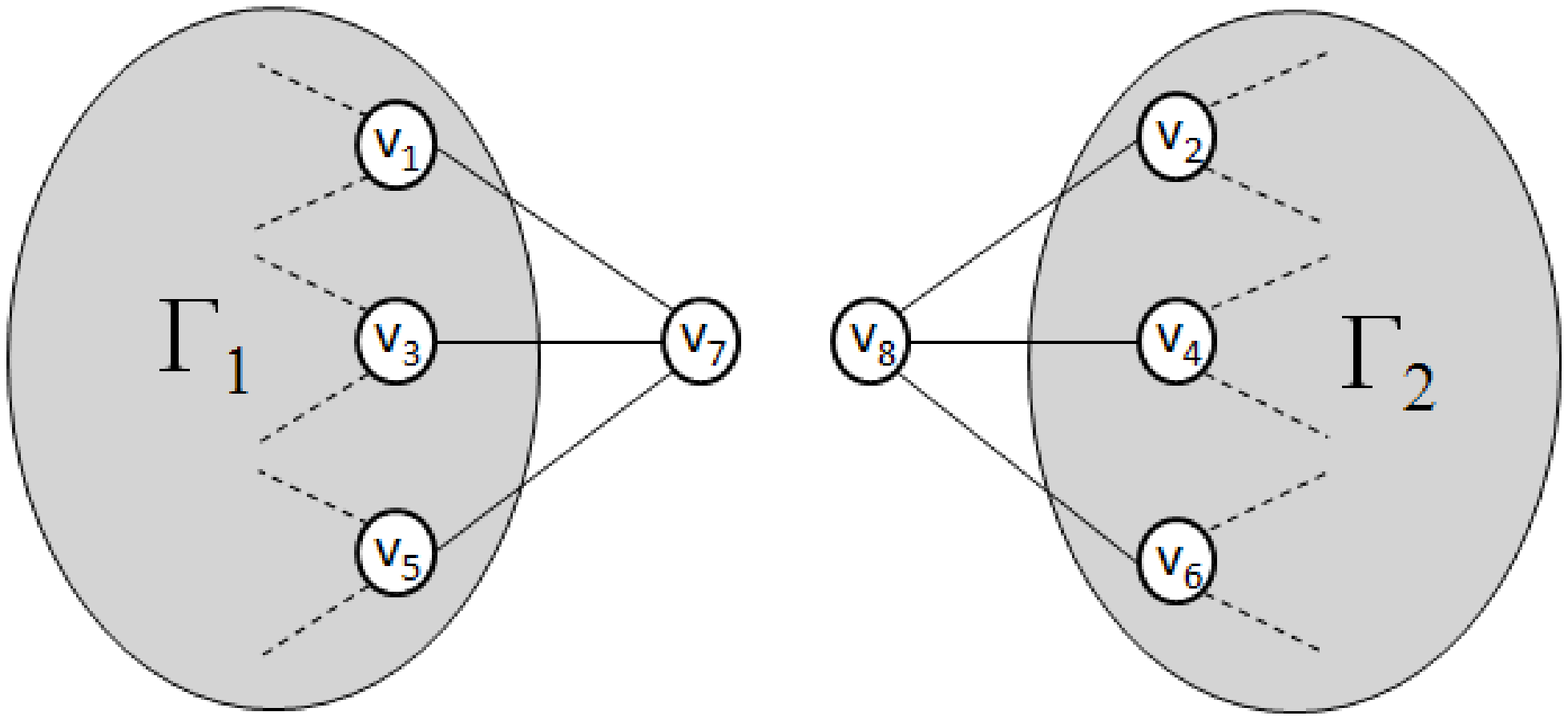}\hspace*{1.0cm}\epsfig{scale=0.35, figure=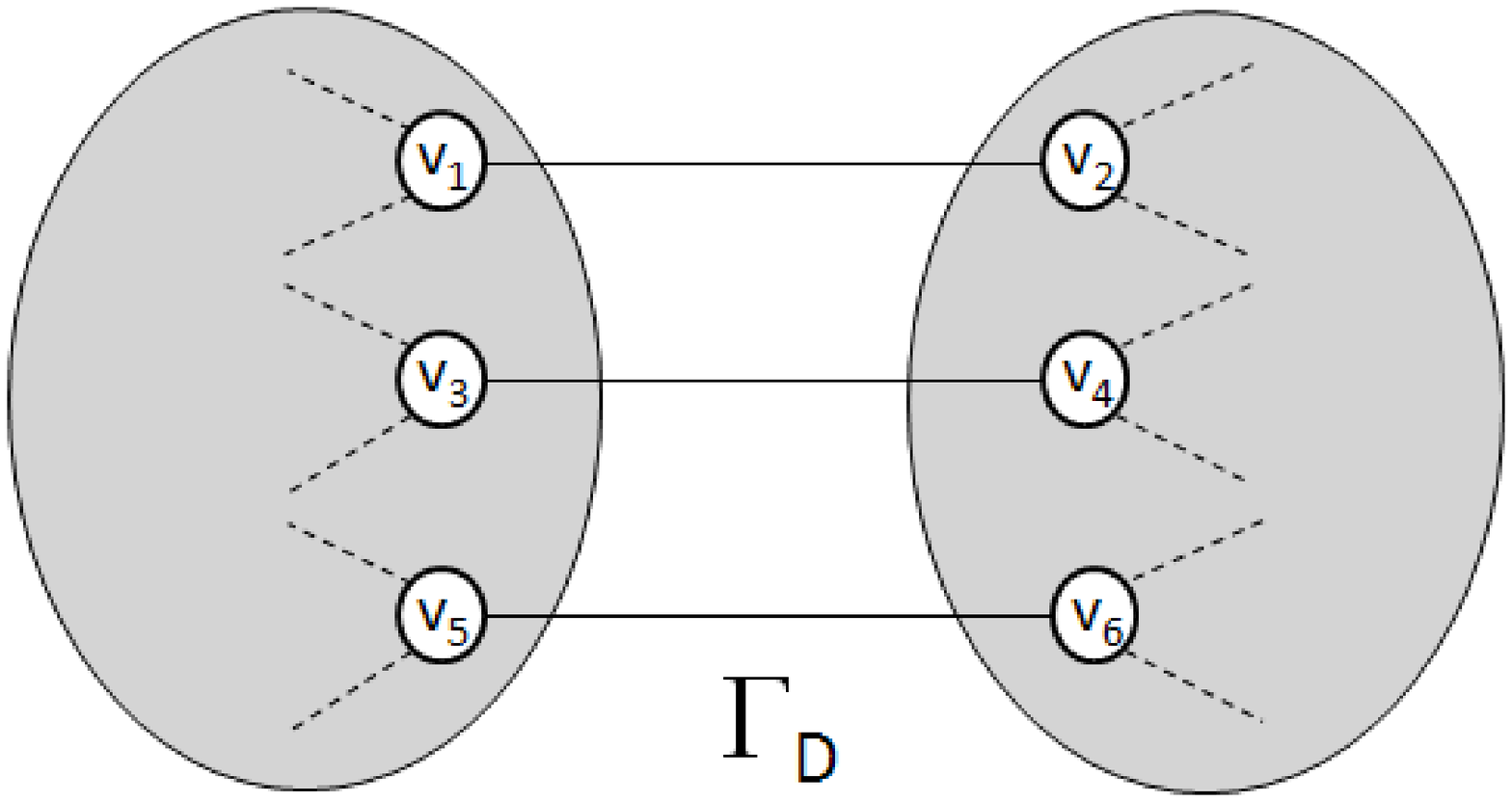}\\
\caption{Graphs $\Gamma_1$, $\Gamma_2$ and $\Gamma_D$ as described
in Proposition
\ref{prop-3cracker}}\label{fig-3_cracker_case_1}\end{center}\end{figure}

Note that any 3-cracker can be obtained from a type 3 breeding operation,
and therefore all 3-crackers are irreducible.

\begin{definition}
We refer to a set of genes that can, through a series of breeding
and parthenogenic operations, be used to produce a descendant
$\Gamma_D$ as {\em ancestor genes} for
$\Gamma_D$.\label{def-ancestor_genes}
\end{definition}

Propositions \ref{prop-1cracker} -- \ref{prop-3cracker} and
Definition \ref{def-ancestor_genes} allow us to derive the main
theorem of this section.

\begin{theorem}Consider any descendant cubic graph $\Gamma_D.$ Then,
\begin{enumerate}\item[(1)] $\Gamma_D$ can be obtained from one or two parents by at least one of the six operations $\mathcal{B}_1(\cdot)$, $\mathcal{B}_2(\cdot)$, $\mathcal{B}_3(\cdot)$, $\mathcal{P}_1(\cdot)$, $\mathcal{P}_2(\cdot)$, $\mathcal{P}_3(\cdot)$.
\item[(2)] Every descendant $\Gamma_D$ has a set of ancestor genes.\end{enumerate}\label{theorem-anygraph}\end{theorem}

\begin{proof} Any descendant graph contains at least one cubic cracker.
Any one of these cubic crackers can be selected, which will be
either a 1-cracker, a $2$-cracker, or a $3$-cracker. Then, from
Propositions \ref{prop-1cracker}, \ref{prop-2cracker} and
\ref{prop-3cracker}, we can obtain $\Gamma_D$, from either one or
two parents, using one of the six operations. Therefore (1) is proved.

Next, consider how $\Gamma_D$ is obtained. If $\Gamma_D$ is
obtained through breeding, it has two parent graphs. If $\Gamma_D$
is obtained through parthenogenesis, it has one parent graph.
These parent graphs may be either genes, or descendants. If any
of them are descendants, then by part (1) they also have one or two
parent graphs each. Inductively, we can continue to consider the
parents of descendants, while recording which operations are used to
produce them, thereby obtaining an {\em ancestral family
tree} of $\Gamma_D$. Once the entire tree is determined, all of
the top nodes are genes, and we can recall the sequence of
operations that produces $\Gamma_D$ from these genes. Therefore
(2) is proved.
\end{proof}

See Figure \ref{fig-ancestral_tree} for an example of an ancestral
family tree for a descendant with 14 vertices.

\begin{figure}[h!]\begin{center}\hspace*{-1cm}\epsfig{scale=0.46, figure=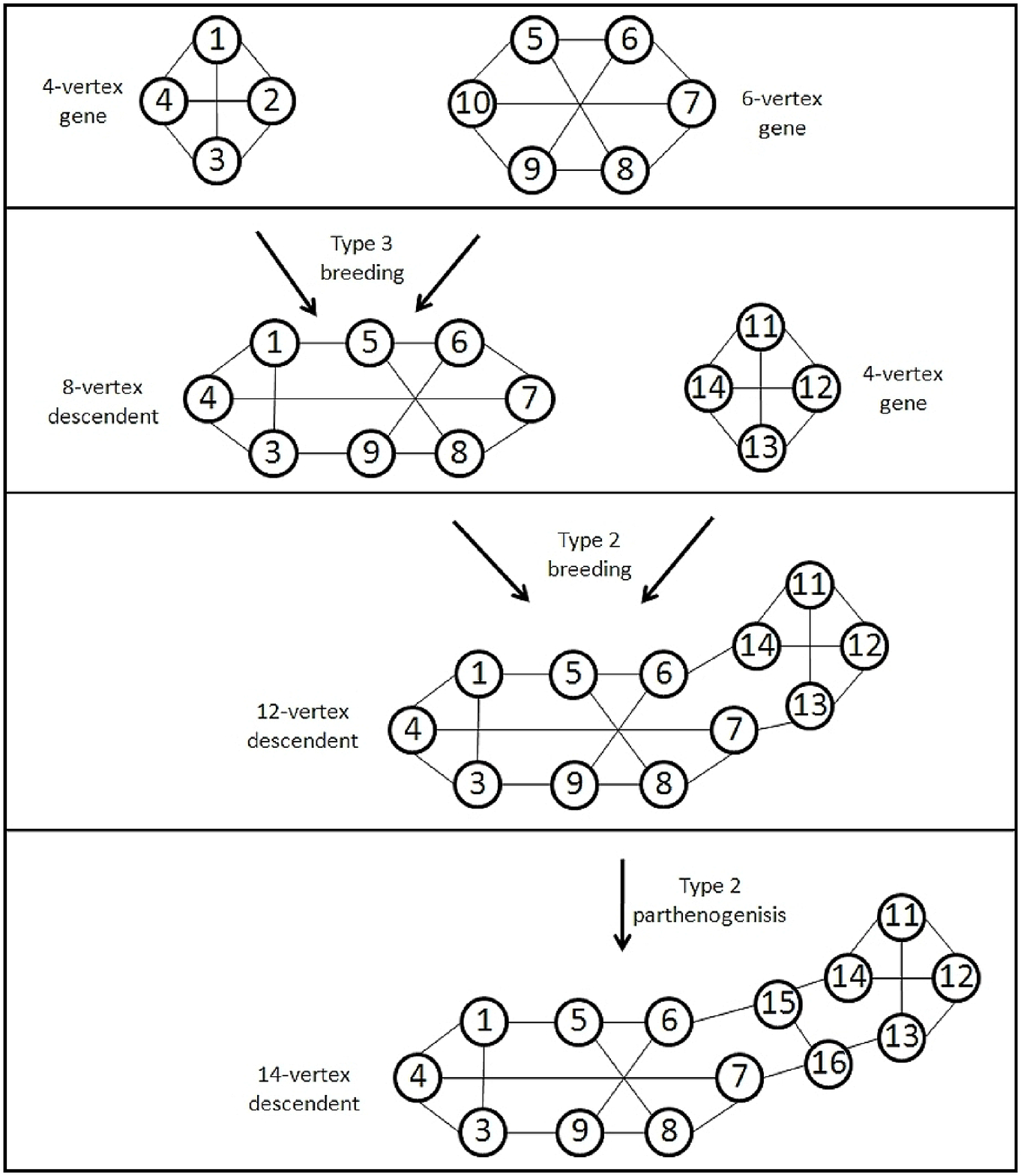}\\
\caption{An ancestral family tree for a 14-vertex
descendant, with three ancestor genes}\label{fig-ancestral_tree}\end{center}\end{figure}

Theorem \ref{theorem-anygraph} indicates that, for any descendant,
we can obtain a set of ancestor genes by first applying an inverse
operation to obtain one or two parents. Then we can apply an
inverse operation on the parent(s) to obtain new parents
(grandparents of the original descendant), and
continue to apply inverse operations until a set of ancestor
genes is obtained.

However, a given descendant may contain only reducible cubic
crackers, which do not permit inverse breeding operations.
In these cases, inverse parthenogenic
operations can be performed, if one of the three parthenogenic
objects are present within the descendant. The removal of such a
parthenogenic object often changes a reducible cracker into an
irreducible cracker which, in turn, permits an inverse breeding
operation to be carried out. The following proposition ensures that,
for any descendant, at least one of the cubic crackers permits
an inverse operation, and furthermore that every reducible cracker
in a given descendant can be changed into an irreducible cracker by
a sequence of inverse parthenogenic operations.

\begin{proposition}If a graph $\Gamma_D$ is a descendant, one of the following must be true.
\begin{enumerate}\item[(1)] $\Gamma_D$ contains at least one irreducible cubic cracker, or
\item[(2)] It is possible to perform a sequence of inverse
parthenogenic operations, each time obtaining a new parent, until
a parent is obtained that contains at least one irreducible cubic
cracker.\end{enumerate}\label{prop-always_irreducible}\end{proposition}

\begin{proof} Since $\Gamma_D$ is a descendant, it contains one or more
cubic crackers. If any of them are irreducible, then (1) is true.
If all the cubic crackers are reducible, then $\Gamma_D$ contains
at least one $1$-cracker or $2$-cracker (as $3$-crackers are always
irreducible). We can therefore select
either a reducible $1$-cracker or a reducible $2$-cracker in
$\Gamma_D$. We will consider both cases separately.

{\bf Case 1:} We select a reducible $1$-cracker $C_1 =
\{(a,b)\}$. Since $\Gamma_D$ is cubic, we know that $a$ must be
adjacent to two more vertices, say $c$ and $d$. Since $C_1$ is
reducible, then without loss of generality, we can assume that
edge $(c,d)$ is present in $\Gamma_D$. Then, the cubicity of
$\Gamma_D$ also ensures that vertices $c$ and $d$ must each be
adjacent to one more vertex, say $e$ and $f$ respectively. Note
that it is possible that $e = f$.

{\bf Case 1.1:} If $e = f$, then the cubicity of $\Gamma_D$
ensures that vertex $e$ must be adjacent to one more vertex, say
$g$ (see Figure \ref{fig-irreducible11}). Then, removing edges
$(a,b)$ and $(e,g)$ from $\Gamma_D$ disconnects a parthenogenic
diamond, which we can remove from $\Gamma_D$ by use of the inverse
type 1 parthenogenic operation
$\mathcal{P}^{-1}_1(\Gamma_D,(a,b),(e,g))$. In the parent graph
$\Gamma_P$, there is a new $1$-cracker comprising edge $(b,g)$.
Note that the new $1$-cracker may also be reducible.

\begin{figure}[h]\begin{center}\hspace*{0cm}\epsfig{scale=0.35, figure=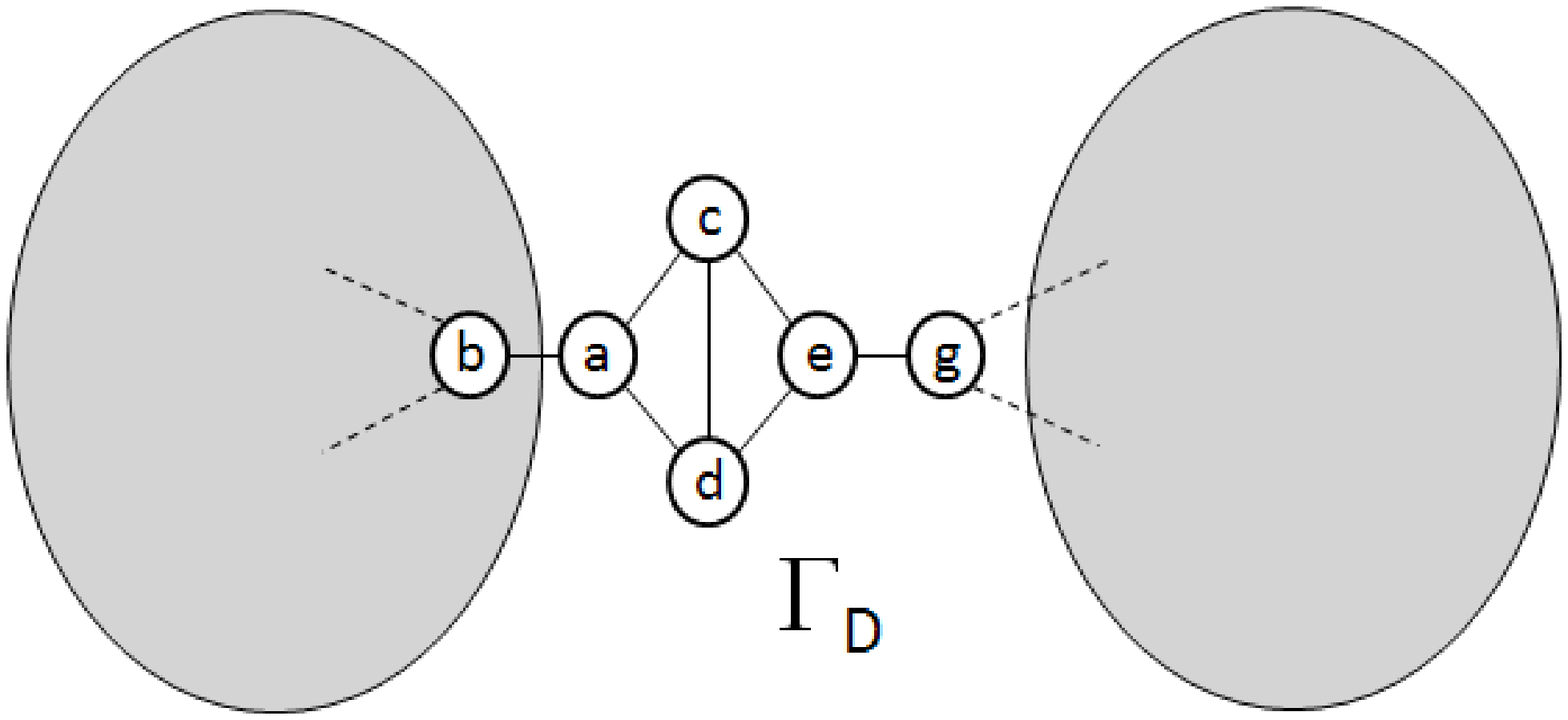}\hspace*{1.0cm}\epsfig{scale=0.35, figure=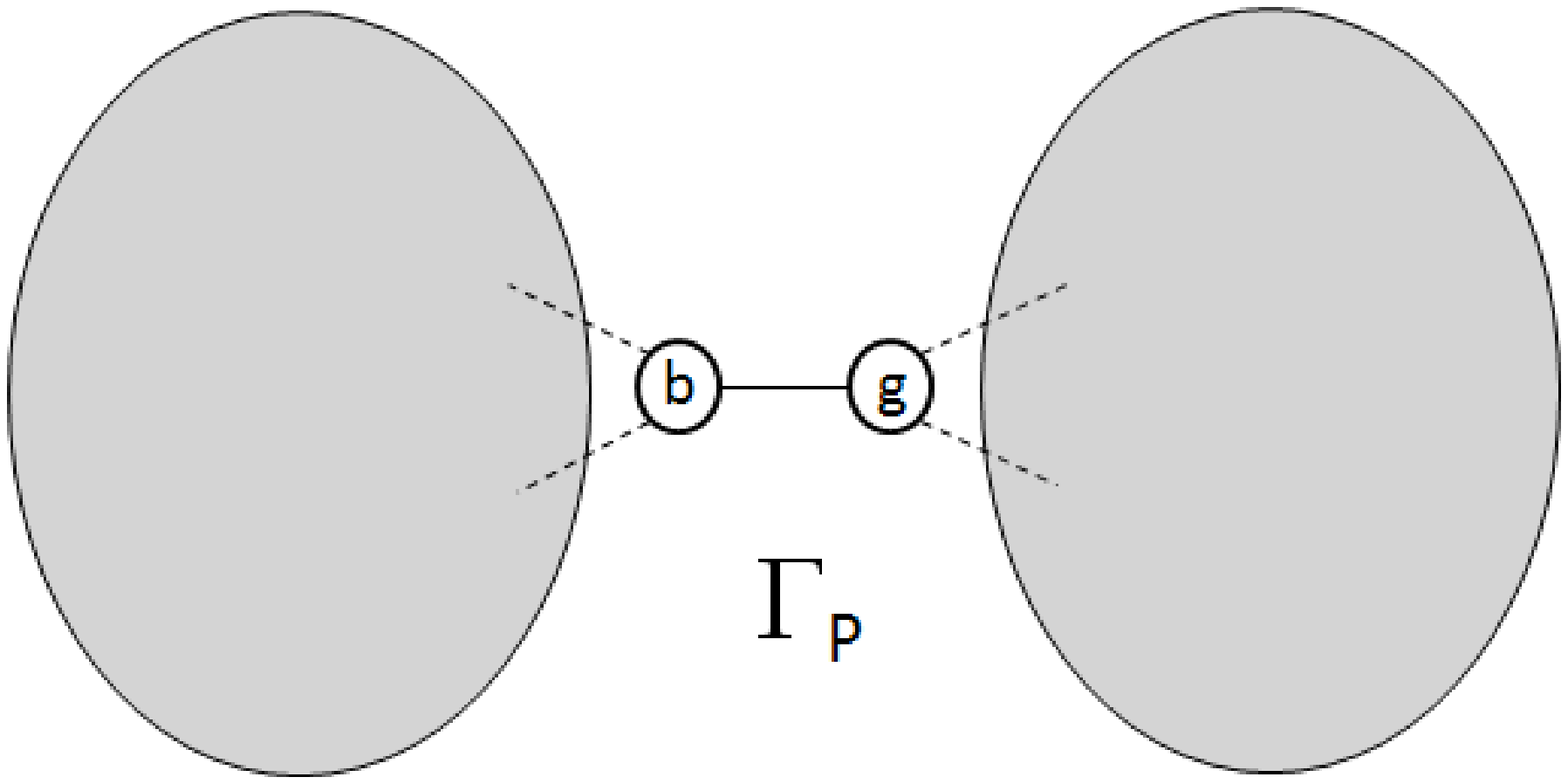}\\
\caption{Graphs $\Gamma_D$ and $\Gamma_P$ as described in Case
1.1}\label{fig-irreducible11}\end{center}\end{figure}

{\bf Case 1.2:} If $e \neq f$, then removing edges $(a,b)$,
$(c,e)$ and $(d,f)$ from $\Gamma_D$ isolates a parthenogenic
triangle (see Figure \ref{fig-irreducible12}), which we can remove
from $\Gamma_D$ by use of the inverse type 3 parthenogenic
operation $\mathcal{P}^{-1}_3(\Gamma_D,a,c,d)$. In the
parent graph $\Gamma_P$, the original $1$-cracker $C_1 =
\{(a,b)\}$ remains. Note that the $1$-cracker may still be
reducible.

\begin{figure}[h]\begin{center}\hspace*{0cm}\epsfig{scale=0.35, figure=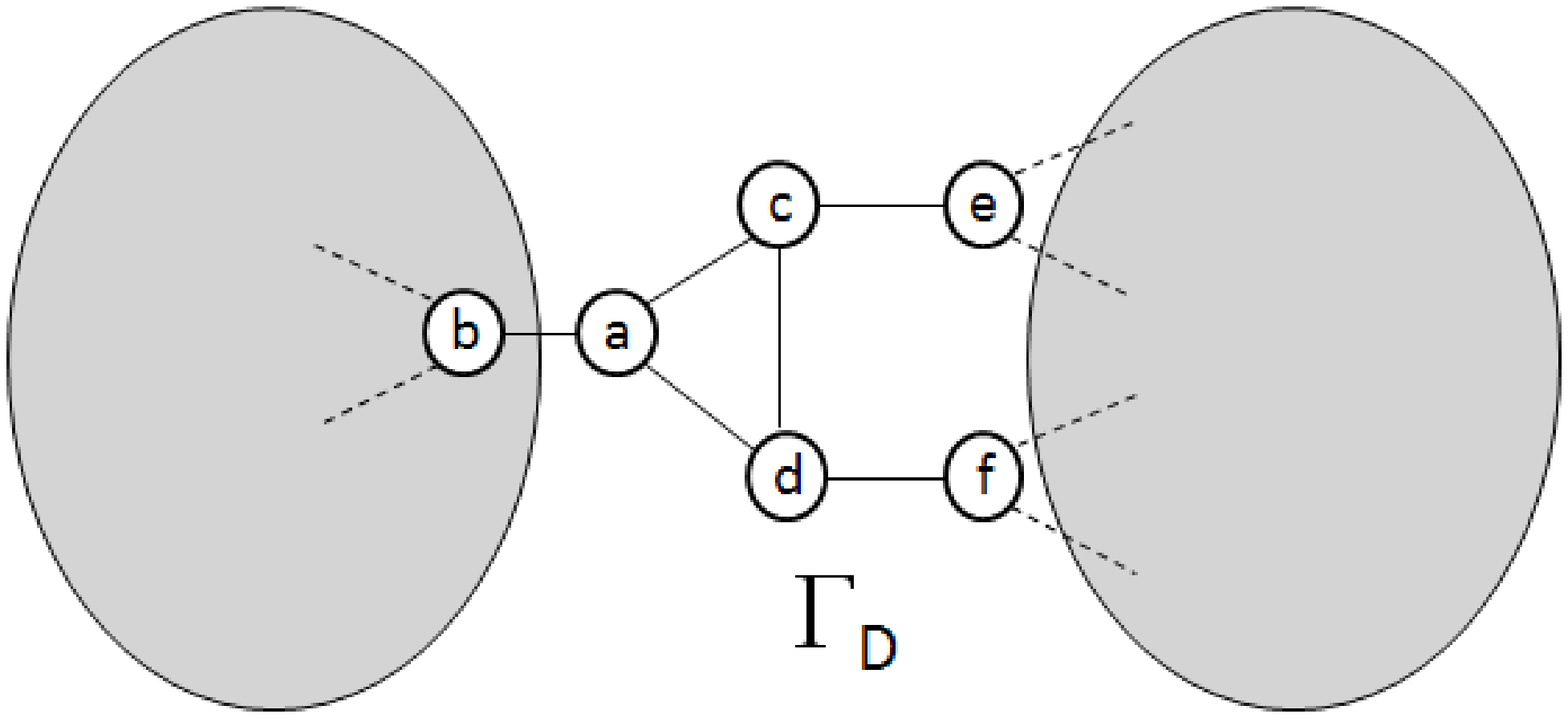}\hspace*{1.0cm}\epsfig{scale=0.35, figure=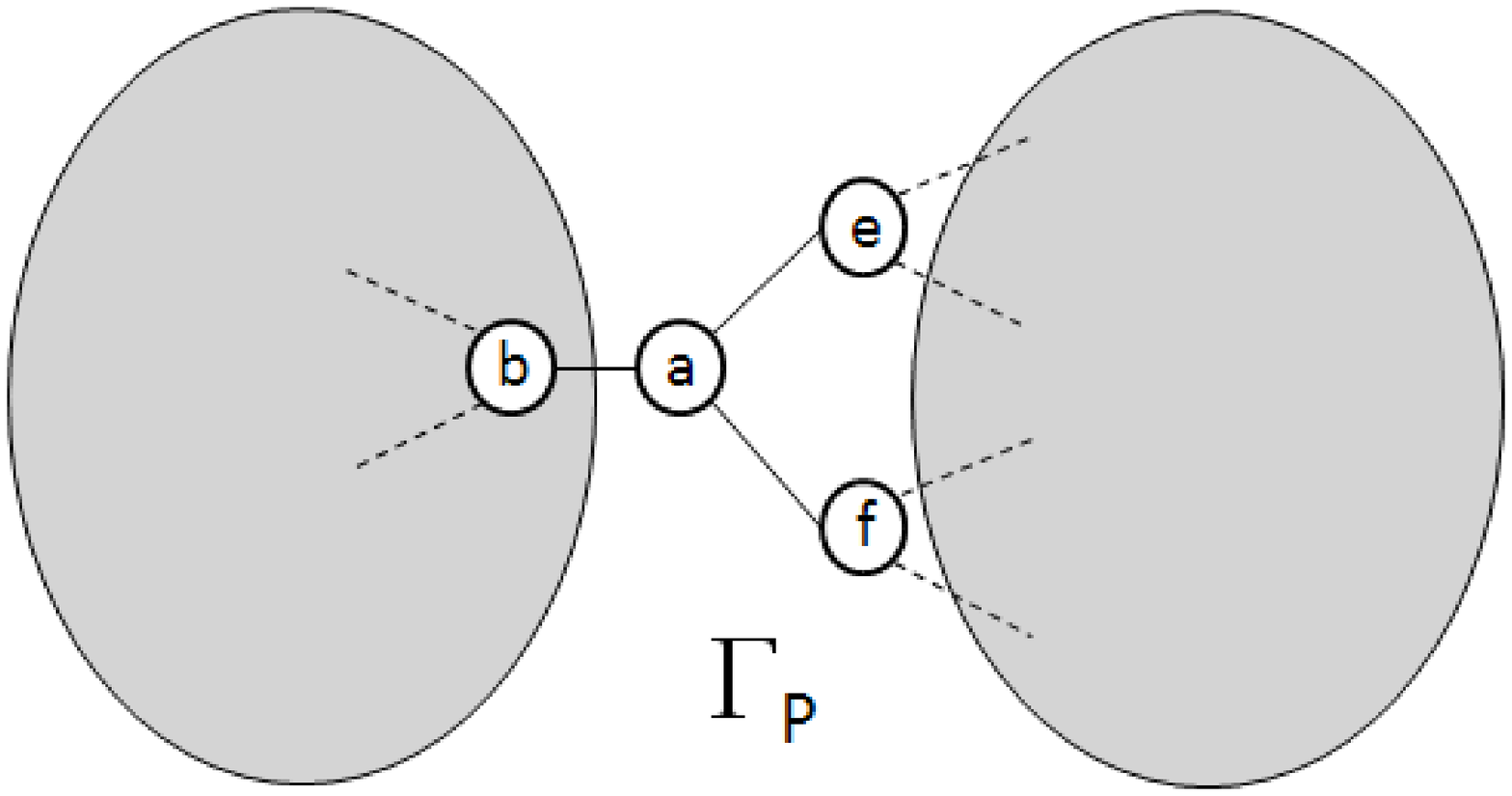}\\
\caption{Graphs $\Gamma_D$ and $\Gamma_P$ as described in Case
1.2}\label{fig-irreducible12}\end{center}\end{figure}

{\bf Case 2:} We select a reducible $2$-cracker $C_2 =
\{(a,b),(c,d)\}$. Since $C_2$ is reducible, then without loss of
generality, we can assume that edge $(a,c)$ is present in
$\Gamma_D$. Then, the cubicity of $\Gamma_D$ also ensures that
vertices $a$ and $c$ must each be adjacent to one more vertex, say
$e$ and $f$, respectively. Note that it is possible that $e = f$.

{\bf Case 2.1:} If $e = f$, then as illustrated in Figure
\ref{fig-irreducible21}, the cubicity of $\Gamma_D$ ensures that
vertex $e$ must be adjacent to one more vertex, say $g$. Then,
removing edges $(a,b)$, $(c,d)$ and $(e,g)$ from $\Gamma_D$
disconnects a parthenogenic triangle, which we can remove from
$\Gamma_D$ by use of the inverse type 1 parthenogenic operation
$\mathcal{P}^{-1}_1(\Gamma_D,a,c)$. In the parent
graph $\Gamma_P$, a $1$-cracker comprising edge $(e,g)$ remains.
Note that this $1$-cracker may still be reducible.

\begin{figure}[h]\begin{center}\hspace*{0cm}\epsfig{scale=0.35, figure=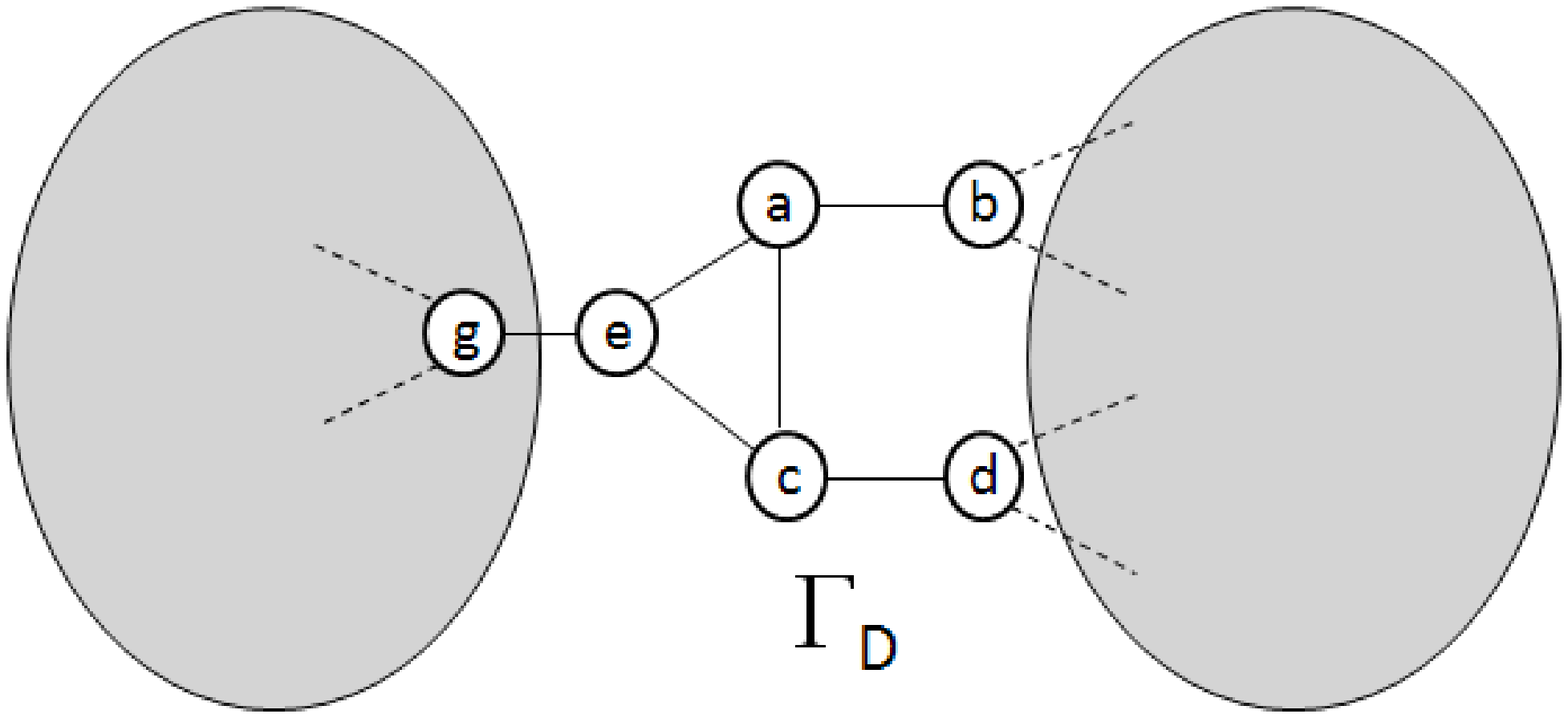}\hspace*{1.0cm}\epsfig{scale=0.35, figure=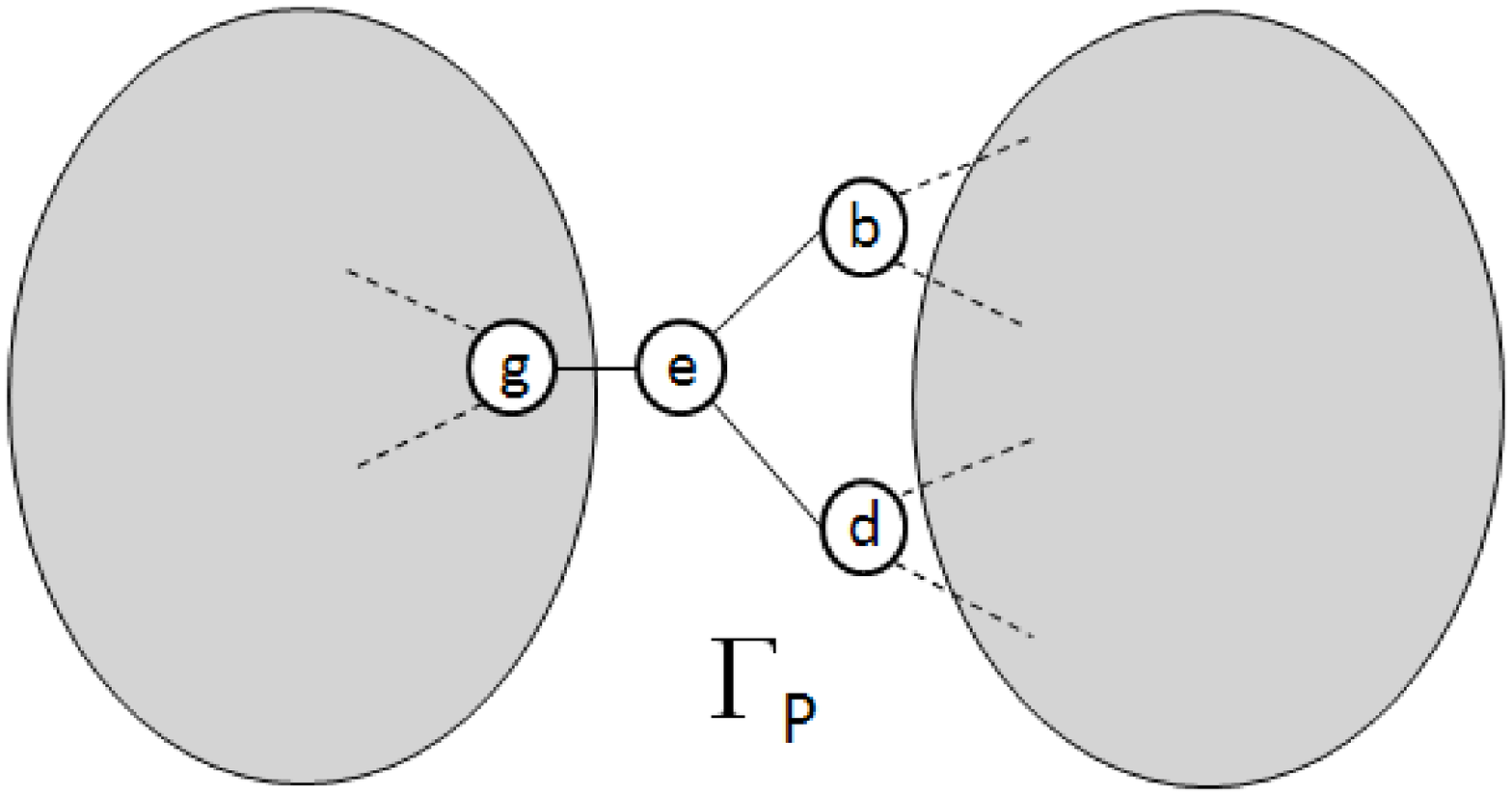}\\
\caption{Graphs $\Gamma_D$ and $\Gamma_P$ as described in Case
2.1}\label{fig-irreducible21}\end{center}\end{figure}

{\bf Case 2.2:} If $e \neq f$, then as illustrated in Figure
\ref{fig-irreducible22}, the removal of edges $(a,b)$, $(c,d)$,
$(a,e)$ and $(c,f)$ from $\Gamma_D$ disconnects a parthenogenic
bridge, which we can remove from $\Gamma_D$ by use of the inverse
type 2 parthenogenic operation
$\mathcal{P}^{-1}_2(\Gamma_D,a,c)$. In the parent graph $\Gamma_P$, there
is a new $2$-cracker comprising edges $(e,b)$ and $(f,d)$. Note
that this $2$-cracker may be reducible.

\begin{figure}[h]\begin{center}\hspace*{0cm}\epsfig{scale=0.35, figure=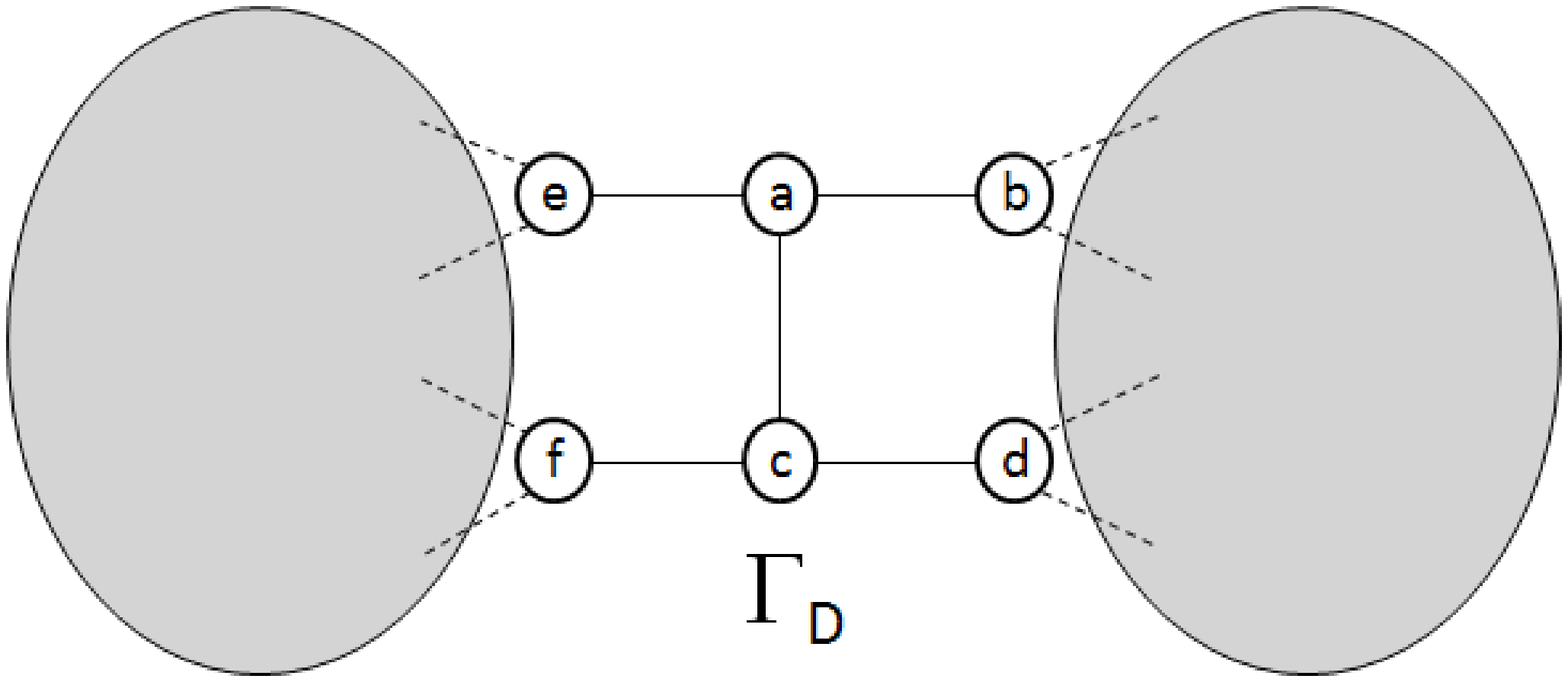}\hspace*{1.0cm}\epsfig{scale=0.35, figure=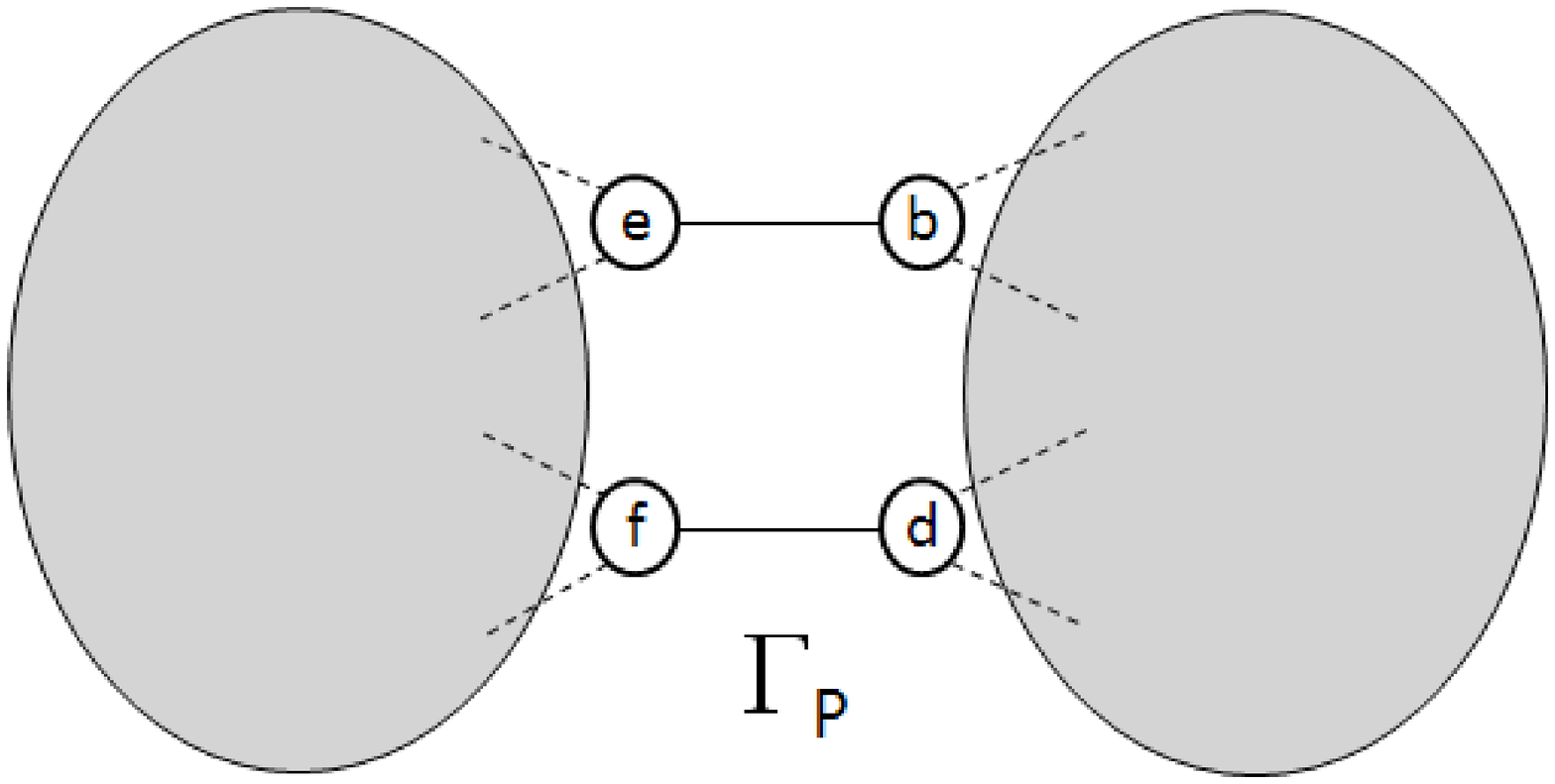}\\
\caption{Graphs $\Gamma_D$ and $\Gamma_P$ as described in Case
2.2}\label{fig-irreducible22}\end{center}\end{figure}

Since in all cases considered, the parent graph contains a cubic
cracker (either a new one introduced by an inverse parthenogenic
operation, or one that remains from the original descendant),
the parent is itself always a descendant. Then,
either (1) is true for this parent, or if not, we can repeat the
above procedure until we obtain a parent for which (1) is true.
Since each inverse parthenogenic process outputs a parent with
fewer vertices than its child, we are guaranteed to eventually
converge to such a case. \end{proof}

\section{Conclusions}

The theory presented above allows us to separate the set $S$ of connected cubic graphs into two distinct and encompassing categories - the comparatively smaller set of genes, that form the basic building blocks of $S$,
and the much larger set of descendants, that inherit a lot of their structure from the genes. Theorem \ref{theorem-anygraph} and Proposition \ref{prop-always_irreducible} give both a proof of existence, and a
guarantee of obtaining a set of ancestor genes for any given descendant. An algorithm to identify a set of ancestor
genes, given a descendant, would be a simple task of identifying all the crackers, surveying each until one
is found that permits an inverse operation, and recursively repeating the process in each parent obtained until only genes remain. Such an algorithm would terminate in polynomial time.

In addition to the aesthetic beauty of rendering the ancestry of cubic graphs as finite sets of smaller cubic graphs, there are some important algorithmic benefits as well. Since descendants inherit much of their
structure from genes, it is possible that a search for graph theoretic properties within a descendant could
be more efficiently recast as a search of the (typically much smaller) genes instead. In this context, Theorem \ref{theorem-anygraph} and Proposition \ref{prop-always_irreducible} give rise to a generic decomposition algorithm that could be applied on most cubic graphs. Since the ancestor genes also lie within the set of connected cubic graphs, any existing algorithms designed for cubic graphs will work for the genes.

For example, experimental evidence indicates that many non-bridge non-Hamiltonian cubic graphs are descendants that contain at least one mutant ancestor gene. In these cases, it is clear that the descendant has inherited the non-Hamiltonicity property from its ancestor mutant gene (or genes). Then, an obvious heuristic for determining Hamiltonicity in a descendant is to identify a set of ancestor genes and determine their Hamiltonicity instead, using whatever state of the art algorithms are available (e.g. see Eppstein \cite{eppstein}). Since determining Hamiltonicity is an NP-complete problem, and therefore the best known algorithms have exponential solving time, such a decomposition represents a large saving in solving time. It is potentially possible that the Hamiltonian cycle problem, already known to be NP-complete even when considering only cubic graphs \cite{gareyjohnson}, could be further refined to requiring the consideration of only genes. If so, the extreme rarity of mutants indicates such an avenue could potentially be fruitful. The inheritance of non-Hamiltonicity, and other such graph theoretic properties, is a subject for future research.

Given that properties can be inherited from a set of ancestor genes, a natural question to ask is how many different sets of ancestor genes a descendant graph might have, and how their various (potentially conflicting) properties may influence the single descendant. The following conjecture, if correct, removes any such confusion.

\begin{conjecture}Any descendant $\Gamma_D$ has a unique set of ancestor genes.\label{conj-unique}\end{conjecture}

To support Conjecture \ref{conj-unique}, we conducted experiments on individual graphs in which we considered all possible ways to decompose the graph into a set of ancestor genes, and verified that each of these approaches gave the same set of ancestor genes. This experiment was performed on all cubic graphs containing up to 18 vertices, and all triangle-free 20-vertex cubic graphs, constituting 143,528 graphs. For these sizes, there can be up to 7! = 5040 possible ways to decompose a graph into ancestor genes. No counterexample to Conjecture \ref{conj-unique} was found among the tested graphs.

If Conjecture \ref{conj-unique} is true it provides, along with Theorem \ref{theorem-anygraph}, a guarantee of existence and uniqueness of ancestor genes for every cubic graph. For such a graph, the above implies that it is either a gene, or there is a unique set of ancestor genes which can be identified in polynomial time, and that these ancestor genes provide the majority of structure in the descendant. If false, Conjecture \ref{conj-unique} may still be true for all but very special classes of descendants. Note that although the conjecture postulates the existence of a unique set of ancestor genes for any given descendant, the order of breeding operations used to obtain the descendant is clearly not unique.

\section*{Acknowledgements}The authors gratefully acknowledge helpful discussions with C.E. Praeger, B.D. McKay and P. Zograf. The research in this manuscript was made possible by grants from the Australian Research Council, specifically by the Discovery grants DP0666632 and DP0984470.

\bibliographystyle{alpha}   

\end{document}